\documentclass[conference]{ieeeconf}

\usepackage{graphicx}
\usepackage{caption}
\usepackage{float}
\usepackage{xcolor}
\usepackage{graphics}
\usepackage{cite}
\usepackage{caption}
\usepackage{comment}
\usepackage{subcaption}
\usepackage{epstopdf}
\usepackage{amsmath,bm}
\usepackage{amssymb}
   
\usepackage{amsthm}
\usepackage{enumerate}
\usepackage{balance}
\usepackage{tabularx}
\usepackage{hyperref}
\usepackage{algorithm}
\usepackage{algorithmic}

\hypersetup{
    colorlinks=true,
    linkcolor=black,
    filecolor=black,
    urlcolor=black,
}
\makeatletter
\newsavebox\myboxA
\newsavebox\myboxB
\newlength\mylenA

\newcommand*\lbar[2][0.75]{%
    \sbox{\myboxA}{$\m@th#2$}%
    \setbox\myboxB\null
    \ht\myboxB=\ht\myboxA%
    \dp\myboxB=\dp\myboxA%
    \wd\myboxB=#1\wd\myboxA
    \sbox\myboxB{$\m@th\overline{\copy\myboxB}$}
    \setlength\mylenA{\the\wd\myboxA}
    \addtolength\mylenA{-\the\wd\myboxB}%
    \ifdim\wd\myboxB<\wd\myboxA%
       \rlap{\hskip 0.5\mylenA\usebox\myboxB}{\usebox\myboxA}%
    \else
        \hskip -0.5\mylenA\rlap{\usebox\myboxA}{\hskip 0.5\mylenA\usebox\myboxB}%
    \fi}
\makeatother
\newtheorem{prop}{Proposition}
\newtheorem{theorem}{Theorem}
\newtheorem{lemma}{Lemma}
\newtheorem{assump}{Assumption}
\newtheorem{coro}{Corollary}
\newcommand{\cA}{\mathcal{A}}
\newcommand{\cQ}{\mathcal{Q}}

\newcommand{\R}{\mathbb R}

	\allowdisplaybreaks

\begin{document}
\title{On the Regret Analysis of Online LQR Control with Predictions}
\author{Runyu Zhang, Yingying Li, and Na Li 
\thanks{This work was supported by NSF CAREER 1553407, AFOSR YIP, ONR YIP. R. Zhang, Y. Li, and N. Li are with the School of Engineering and Applied Sciences, Harvard University, 33 Oxford Street, Cambridge, MA 02138, USA (email: runyuzhang@fas.harvard.edu, yingyingli@g.harvard.edu, nali@seas.harvard.edu).}
}

\maketitle
\begin{abstract}
    In this paper, we study the dynamic regret of online linear quadratic regulator (LQR) control with time-varying cost functions and disturbances. We consider the case where a finite look-ahead window of cost functions and disturbances is available at each stage. The online control algorithm studied in this paper falls into the category of model predictive control (MPC) with a particular choice of terminal costs to ensure the exponential stability of MPC. It is proved that the regret of such an online algorithm decays exponentially fast with the length of predictions. The impact of inaccurate prediction on disturbances is also investigated in this paper.
\end{abstract}
\section{Introduction} \label{sec:intro}

Consider a classical finite-horizon discrete-time linear quadratic regulator (LQR) problem:
\begin{equation}\label{equ: lqr}
    \begin{aligned}
    \min_{u_1, \dots, u_T}&\   \sum_{t=1}^{T-1} (x_t^\top Q_tx_t + u_t^\top R_tu_t) + x_T^\top Q_Tx_T\\
    \text{s.t. } & \ x_{t+1} = Ax_t + B_uu_t + B_dd_t, \quad t\geq 1,
\end{aligned}
\end{equation}
where the cost function parameters $Q_t, R_t$ and the system disturbances $d_t$ are time-varying, while the system parameters $A, B_u, B_d$ are time-invariant. It is well-known that the optimal control input to \eqref{equ: lqr} at time step $t$ requires the information of all the future, i.e. $\{Q_{\tau},R_{\tau}, d_{\tau}\}_{\tau\geq t}$ \cite{Bersikas_optimal_control_dynamic_programming}. However, in most real-world applications, e.g. autonomous driving \cite{kim2014mpc}, energy systems \cite{kouro2008model}, date center management \cite{lazic2018data}, it is  impractical for an decision maker to acquire all the (accurate) future information beforehand. Instead, the decision maker may only have access to some predictions for the near future and the predictions can be inaccurate.  Hence, this calls for the study of  \textit{online LQR problem 
with limited and inaccurate  predictions of the future}. 
Specifically, this paper considers the following online LQR problem: at each time step $t$, the decision maker receives cost predictions $\{Q_{i\mid t}, R_{i\mid t}\}_{i=t}^{t+W} $ and disturbance predictions $\{d_{i\mid t}\}_{i=t}^{t+W}$ for the next $W$ time steps. For simplicity, we only consider inaccurate disturbance predictions and assume cost predictions are accurate. The goal of online LQR is to minimize the total cost in \eqref{equ: lqr} by only leveraging the predictions and the history.

Among all the online control algorithms that leverage predictions, perhaps model predictive control (MPC) is the most popular one. Although MPC has been intensively studied both for linear systems and nonlinear systems \cite{diehl2010, ellis2014,amrit2011economic, grune2013economic, angeli2011average, grune2014, grune2020}, most studies focus on asymptotic performance such as stability or convergence to some optimal state. Motivated by the aforementioned applications, there is an increasing need to understand the non-asymptotic performances of MPC such as the cost difference compared to the optimal cost over a finite time-horizon. Though there are recent online control papers, e.g.,  \cite{abbasi2014tracking,cohen2018online,agarwal2019online}, that study the non-asymptotic behavior of online control algorithms, they do not consider predictions, that is, the controller has to take an control action without any knowledge of future  $\{Q_{\tau},R_{\tau}, d_{\tau}\}_{\tau\geq t}$ at time $t$.

In online learning community, on the contrary, there are many papers on the non-asymptotic performance analysis, where the performance is usually measured by regret, e.g., static regrets \cite{hazan2019introduction,shalev2011online}, dynamic regrets \cite{jadbabaie2015online}, etc. But most papers either do not system dynamics or predictions \cite{rakhlin2013online, chen2015online, badiei2015online, chen2016using} or only consider special simple dynamics \cite{li2018online, goel2019online} or simplified prediction models where $W$-step ahead predictions are accurate without errors \cite{li2018online,li2019}. 

 The setting considered in this paper is closest to the recent papers \cite{li2019,li2020, yu2020, goel2020power}. \cite{li2019,li2020} consider a linear dynamical system with time-varying cost functions but no disturbances. They focus on gradient-based online control rather than the more commonly used MPC approach. On the other hand, \cite{goel2020power,yu2020}, study the non-asymptotic behavior of MPC algorithm, but they only consider time invariant cost function and accurate disturbance predictions. But one interesting message from \cite{yu2020} is that MPC turns out to be optimal (or nearly optimal) in stochastic settings (or in adversarial settings) with respect to (dynamic) regrets. This further motivates us to look into the regret analysis of MPC for the setting with time-varying cost functions and inaccurate disturbance predictions.

\textbf{Contribution.} In this paper, we provide an explicit upper bound for the performance of MPC (Theorem~\ref{thm:regret-analysis-main}) in terms of the dynamic regret: the online cost minus the optimal cost in hindsight. The MPC studied in this paper follows the standard MPC framework with a particular choice of terminal cost which ensures the exponential stability of MPC. Our regret bound consists  of two parts: the first part decays exponentially with the prediction window $W$ and the second part increases with the  prediction errors. The first part indicates the benefits of having more predictions and the second part reflects the negative impact of inaccurate predictions. When there are no prediction errors, the second part is zero and  the regret bound decays exponentially with $W$. Further, in the second part of our regret bound, the errors of long-term predictions play an exponentially diminishing effect, i.e. the impact of the error of $d_{t+k\mid t}$ decays exponentially with $k$ on the regret bound. This indicates that MPC (implicitly) focuses more on the short-term predictions and less on long-term ones, which is desirable in most cases  since long-term predictions usually suffer from poor quality. Further, this suggests that our regret upper bound provides useful guidelines in choosing the prediction window $W$ used in MPC in the face of inaccurate predictions. 


To develop our regret bound, we provide a general regret formula for any online algorithms for online LQR problems, which is a quadratic function of the differences between the online control actions and the optimal control actions. Our formula is established by leveraging a cost difference lemma  in the (Markov decision processes) literature \cite{schulman2015,fazel2018} and the special properties of LQR. Our formula greatly relieves the difficulty of non-asymptotic regret analysis of MPC.  Furthermore, the formula can be applied to other online control algorithms  and thus can be viewed as a contribution on its own merit.

\textbf{Notations:} The norm $\|\cdot\|$ refers to the $L_2$ norm for both vectors and matrices. $\lambda_{\min}(A)$ denotes the minimum eigenvalue of matrix $A$, and $\lambda_{\max}(A)$ denotes the maximum eigenvalue of $A$. For any symmetric matrices $A$ and $B$, we write $A \prec B$ if $B-A$ is positive definite.  

\section{Problem Setup and Preliminaries}
\subsection{Problem Formulation: Online LQR}
As stated in Section~\ref{sec:intro}, we consider an online linear quadratic regulator (LQR) problem with process noises/disturbances. The system dynamics is provided by:
\begin{equation}\label{equ: linear dynamics}
    x_{t+1} = Ax_t + B_uu_t + B_dd_t, \quad t\geq 1
\end{equation}
where the initial state $x_1 \in \R^n$ is fixed, $x_t \in \R^n$ and $ u_t \in \mathbb{R}^{n_u}$ denote the state and control input at stage $t$ respectively, and  $d_t \in \mathbb{R}^{n_d}$ denotes the process noise/disturbance. We consider time-varying quadratic costs at each stage $t$, i.e. $x_t^\top Q_t x_t + u_t^\top R_t u_t$. 
The total cost over $T$ stages is defined as
\begin{equation*}
    J(\mathbf{x}, \mathbf{u}) = \sum_{t=1}^{T-1} (x_t^\top Q_tx_t + u_t^\top R_tu_t) + x_T^\top Q_Tx_T, \label{eq:JLQR}
\end{equation*}
where $\mathbf{x} = [x_1^\top, \dots, x_T^\top]^\top, \mathbf{u} = [u_1^\top, \dots, u_{T-1}^\top]^\top$,   $Q_T$ represents the terminal cost, and we define $R_T=0$ for notational simplicity.




The control objective is to design control input $u_t$ at each $t$ to minimize the total cost $J(\mathbf{x},\mathbf{u})$. However, the optimal control at time $t$ requires the information of all the future cost functions and disturbances, i.e. $\left\{d_i,Q_i, R_i\right\}_{i=t}^T$ (see e.g. Proposition~\ref{prop:optimal-policy}), which may  not be practical in real-world applications.  Nevertheless, some predictions are usually available beforehand, especially for the near future.  In this paper, we consider that the predictions of  the next $W$ stages are available, i.e. $\{d_{i|t}, Q_{i|t}, R_{i|t}\}_{i=t}^{t+W}$, where $d_{i\mid t}$ denotes the prediction of the disturbance $d_i$ at stage $t$, and the same applies to $ Q_{i|t}, R_{i|t}$. The prediction information can be inaccurate and it is worth discussing the impact of the prediction errors. For this paper, we focus on the  predictions errors of the disturbances and denote  $e_{i|t} := d_{i|t} - d_t$ as the prediction error of $d_{i\mid t}$. Since the prediction $d_{i\mid t}$ is received at stage $t$, which is $(i-t)$ stages  before stage $i$, we also call $e_{i|t}$ as the $(i-t)$-step prediction error of $d_i$. For simplicity, we 
omit the prediction errors of the cost matrices  by considering accurate $W$-stage cost predictions: $Q_{i\mid t}=Q_i$ and $R_{i\mid t}=R_i$ for $t\leq i \leq t+W$.\footnote{Ideally, we would like to also consider inaccurate prediction on $Q_i, R_i$, but due to some technical difficulty in analyzing the performance, it is left as our future work. Nevertheless, our  setting still finds applications, e.g., when the cost function is set according to financial contracts or cost planning steps-ahead. However, disturbances are often due to volatile nature such as wind. Thus allowing inaccurate predictions greatly broadens the applications of previous settings studied in \cite{li2019,yu2020,goel2019online}.}

In summary, the online LQR considered in this paper is described as follows: Assuming $A, B_u, B_d$ is known apriori, at each step $t = 1, 2, \dots,$
\begin{itemize}
    \item the controller observes state $x_t$ and receives predictions $\{d_{i|t}, Q_i, R_i\}_{i=t}^{t+W}$;
    \item  the controller  implements  $u_t$ based on the predictions $\{d_{i|t}, Q_i, R_i\}_{i=t}^{t+W}$ and the
    history  $\{x_i, d_i, Q_i, R_i\}_{i=1}^t$ and suffers the cost $x_t^\top Q_t x_t + u_t^\top R_t u_t$;
    \item the system evolves to the next state $x_{t+1}$ by  \eqref{equ: linear dynamics} under the real disturbance $d_t$.
\end{itemize}

Our goal is to design an online algorithm to reduce the total cost by  exploiting the online available information, i.e. the predictions and the history. For example, consider an online control algorithm denoted by $\bm \pi=\{ \pi_1, \dots, \pi_{T-1}\}$, where $\pi_t$ represents the policy at step $t$. Notice that $\pi_t$ only has access to  the online available information at $t$, i.e., the control action  $u_t$ at time $t$ is determined by
\begin{equation} \label{eq:online_pi}
    u_t = \pi_t(\underbrace{\{x_\tau\}_{\tau=1}^t, \{d_\tau, Q_\tau,R_\tau\}_{\tau=1}^{t}}_{\text{history}}, \underbrace{\{d_{i|t}, Q_i, R_i\}_{i=t}^{t+W}}_{\text{predictions}}), 
\end{equation}
We measure the  performance of the online algorithm $\bm \pi$ by \textit{dynamic regret}, which compares the total cost of $\bm \pi$  with the optimal total cost $J^*$ in hindsight, that is, 
\begin{equation*}
   \text{Regret}(\bm \pi)= J(\mathbf{x}^{\bm \pi},\mathbf{u}^{\bm \pi}) - J^*.
\end{equation*}
where $x_t^{\bm \pi}, u_t^{\bm \pi}$ denote the state and action at step $t$ generated by  the online algorithm $\bm \pi$. Let $\bm \pi^*=\{\pi_1^*, \dots, \pi_{T-1}^*\}$ denote the optimal controller in hindsight that yields the optimal cost $J^*$. Here $\pi_t^*$ is the optimal policy at each step $t$, which will be further discussed in Section~\ref{sec:offlineLQR}. 
 
 Dynamic regret is a commonly used performance metric in the literature \cite{li2019,yu2020}. The benchmark of the dynamic regret defined above is optimal \textit{time-varying} policies. Notice that that another popular regret notion is the \textit{static regret}, whose benchmark is the optimal \textit{time-invariant} policy which is a weaker benchmark because the optimal control for finite-time horizon time-varying LQR (\ref{eq:JLQR}) is time-varying. 

Throughout the paper, we consider the following assumptions on the dynamics and  the cost matrices, which are standard assumptions in the literature.

\begin{assump}\label{assump:stabilizability}
The pair $(A,B_u)$ is stabilizable. All pairs $(A, Q_t)$ are detectable. 
\end{assump}
\begin{assump}\label{assump:bounded-S-R}
There exist positive definite matrices $Q_{\min},$ $ Q_{\max}, R_{\min}, R_{\max}$ such that $Q_t$  for any $1\leq t \leq T$ and $R_t $ for any $1\leq t \leq T-1$ satisfy
$$0\prec Q_{\min}\preceq Q_t \preceq Q_{\max}, \quad0\prec R_{\min} \preceq R_t \preceq R_{\max}.$$
\end{assump}
\subsection{Preliminaries: Optimal Offline Controller} \label{sec:offlineLQR}
Here we provide some preliminaries on the optimal offline LQR \cite{Bersikas_optimal_control_dynamic_programming,anderson2007}  which will be used in analyzing the regret of our online controller.  Throughout the paper we will use the notation $F_{Q,R}(\cdot)$
Ricatti iteration for standard LQR given the system dynamics $(A,B_u)$:
$$F_{Q,R}(P) := Q + A^\top PA - A^\top P B_u(R+B_u^\top PB_u)^{-1}B_u^\top PA$$

We use $P_t$ to denote the optimal cost-to-go matrix for the standard LQR problem. $P_t$ is calculated through:
\begin{equation} \label{eq:iterative-Riccati}
\left\{
\begin{array}{l}
     P_t = F_{Q_t, R_t}(P_{t+1}),\quad  1\le t \le T-1\\
     P_T = Q_T
\end{array}
\right.
\end{equation}
The optimal-cost-to-go is given by, for any $x$,
\begin{align*}
    &\quad x^\top P_t x =\\ &\min_{\{u_t\}_{t\!=1}^{T\!-1}} \left\{\sum_{i=t}^{T-1}(x_i^\top Q_ix_i + u_i^\top R_iu_i) + x_T^\top Q_Tx_T \mid x_t = x\right\}
\end{align*}

The optimal control gain of the LQR problem is:
\begin{equation}\label{eq:K_t-LQR}
\begin{split}
    K_t=(R_t+B_u^\top P_{t+1}B_u)^{-1}B_u^\top P_{t+1}A,
\end{split}
\end{equation}

The definition of $P_t, K_t$ depends on the matrices series $\{Q_i, R_i\}_{i=t}^{T}$. Throughout the paper we might use different series of matrices to compute its corresponding $P_t, K_t$'s, thus we rewrite the variables $P_t, K_t$ as functions, i.e.,
$$P_t(\{Q_i, R_i\}_{i=t}^{T-1}, Q_T), ~K_t(\{Q_i, R_i\}_{i=t}^{T-1}, Q_T)$$ 
to denote the optimal cost-to-go matrix and control gain for the standard LQR problem, given the sequence of stage cost matrices $\{Q_i, R_i\}_{i=t}^{T-1}$ and terminate cost $Q_T$. If not stated otherwise, we will use the short notation $P_t, K_t$ to denote the cost-to-go matrix and control gain for the offline setting.

We introduce the variable $P_{\max}$ to denote the solution of the following \textit{discrete time Riccati equation} (DARE):
\begin{equation}\label{eq:P_max}
\begin{split}
    P_{\max} = F_{Q_{\max},R_{\max}}(P_{\max})
\end{split}
\end{equation}

Furthermore, we define the state transition matrix as:
\begin{equation*}
    \Phi(t, t_0) := \left\{
    \begin{array}{ll}
        (A-B_uK_{t-1})\cdots(A-B_uK_{t_0}), & t>t_0 \\
        I, & t = t_0 
    \end{array}
    \right.
\end{equation*}
It is known that the optimal control action is a linear combination of current state $x_t$ and all future disturbances \cite{anderson2007,goel2020power}.
\begin{prop}{\cite{goel2020power}}\label{prop:optimal-policy}
The optimal policies $\{\pi_t^*\}_{t=1}^{T-1}$ that minimize $J(\mathbf{x},\mathbf{u})$ can be written as:
\begin{equation}\label{eq:optimal-policy}
\begin{split}
     \bm \pi^*:u_t \! =\pi_t^*(x_t,d_t,\ldots,d_{T-1})=\! -K_tx_t - \sum_{i=t}^{T-1} K_t^{d,i} B_dd_i,
\end{split}
\end{equation}
where for $i=t,\ldots,T-1$
\begin{equation}\label{eq:K_t^i}
    K_{t}^{d,i} = (R_t + B_u^\top P_{t+1}B_u)^{-1}B_u^\top\Phi(i+1, t+1)^\top P_{i+1}.
\end{equation}
\end{prop}
 Furthermore, the exponential stability of LQR can also be established.
\begin{prop}{(Exponential Stability of the Optimal Controller \cite{anderson2007}, Sec. 3.2)}
\label{prop:stability-LQR}
The state transition matrix for finite time horizon optimal LQR control is exponentially stable, i.e. 
\begin{equation*}
    \|\Phi(t,t_0)\| \le \tau \rho^{t-t_0},
\end{equation*}
where
$$\tau = \sqrt{\frac{\lambda_{\max}(P_{\max})}{\lambda_{\min}(Q_{\min})}}, \quad \rho = \sqrt{1-\frac{\lambda_{\min}(Q_{\min})}{\lambda_{\max}(P_{\max})}}.$$
\end{prop}
This leads to the exponential decaying properties of $K_t^{d,i}$.
\begin{coro}\label{coro:bound k_t^i}
The matrices $K_t^{d,i}$ defined in \eqref{eq:K_t^i} satisfy
\begin{equation*}
    \|K_t^{d,i}\| \le \frac{\tau\|B_u\|\lambda_{\max}(P_{\max})}{\lambda_{\min}(R_{\min})}\rho^{i-t},\quad i\ge t.
\end{equation*}
\end{coro}

In \cite{anderson2007}, their proof of exponential stability is for continuous time infinite horizon case, but the proof technique is quite similar for discrete time and finite horizon setting. 
Proposition \ref{prop:optimal-policy} and Corollary \ref{coro:bound k_t^i} suggest that disturbances from far future do not have too much impact on current control action. The exponential decaying of $\|K_t^{d,i}\|$ implies that the weight on disturbances in the far future will be fairly small. This property enables the possibility of finding a relatively good controller using only limited predictions.

\section{Model Predictive Control}
Model predictive control (MPC) is perhaps the most common control policy for situations where predictions are available \cite{muller2017,ellis2014,ferramosca2010,amrit2011economic}. Generally speaking, an MPC algorithm with $W$-step look-ahead window with stage cost $c_t(x,u)$ and terminal cost $T(x)$ is defined as follows:
\begin{equation} \label{eq:MPC-general}
\begin{split}
    \min_{\{u_k\}_{k=t}^{t+W}} &  \sum_{k=t}^{t+W} c_k(x_k, u_k) + T(x_{t+W+1})\\
    s.t. &   \quad x_{k+1} = f_k(x_k, u_k, d_k)
    \end{split}
\end{equation}
where $x_{k+1} = f_k(x_k, u_k, d_k)$ is the system dynamics. At each time step $t$, MPC solves the above equation and implements  $u_t$ output from the solver. Specifically in our setting, the stage cost functions are given as $c_t(x,u) = x^\top Q_tx + u^\top R_tu$. The terminal cost is chosen to be $T(x) = x^\top P_{\max}x$ for ensuring the stability of the algorithm. The MPC algorithm is given as,
\begin{equation}\label{eq:MPC}
    \begin{split}
    \min_{\{u_k\}_{k=t}^{t+W}} & \sum_{k=t}^{t+W}(x_k^\top Q_kx_k + u_k^\top R_ku_k) + x_{t+W+1}^\top P_{\max}x_{t+W+1}\\
    s.t. &\quad x_{k+1} = Ax_k + B_uu_k + B_dd_{k|t}, k = t, \dots, t+W
\end{split}
\end{equation}
 
Though MPC is a well studied topic in the control community, most results focus on asymptotic analysis, such as stability and convergence to the optimal action as $t\rightarrow \infty$. Non-asymptotic analysis, on the other hand, such as dynamic regret analysis, which also takes transient behavior of the dynamic into consideration, are less studied. However, the recent growing research in online and reinforcement learning calls for more study in characterizing the non-asymptotic performance of MPC.

In order to analyze the dynamic regret of the MPC, we first seek a different representation of the online algorithm which allows us to compare the online MPC policy and the optimal offline policy in \eqref{eq:optimal-policy}. The key observation is that Proposition \ref{prop:optimal-policy} can also be applied to solve \eqref{eq:MPC} for MPC which allow us to represent the MPC in a similar form as in \eqref{eq:optimal-policy}, 
\begin{equation}\label{eq:MPC-rewrite}
    \textup{MPC}: ~u_t^\textup{MPC} = -\lbar{K}_tx_t^\textup{MPC} - \sum_{i=t}^{t+W}\lbar{K}_t^{d,i}B_dd_{i|t}
\end{equation}
Here $\lbar{K}_t,\lbar{K}_t^{d,i}$ are constructed in a similar manner as $K_t, K_t^{d,i}$. First define:
\begin{equation}\label{eq:bar-K-P}
\begin{split}
    \lbar{P}_{t+\tau|t} &:= P_{t+\tau}(\{Q_i,R_i\}_{i=t+\tau}^{t+W}, P_{\max})\\
    \lbar{K}_{t+\tau|t} &:= K_{t+\tau}(\{Q_i,R_i\}_{i=t+\tau}^{t+W}, P_{\max})\\
    &= (R_{t+\tau} + B_u^\top \lbar{P}_{t+\tau+1|t}B_u)^{-1}B_u^\top\lbar{P}_{t+\tau+1|t}A
\end{split}
\end{equation}
which are the cost-to-go and control gain matrices given the $W$ steps ahead prediction $\{Q_i,R_i\}_{i=t}^{t+W}$ and the terminal cost $P_{\max}$ at step $t+W+1$.

Further define the `predicted' state transition matrix at step $t$:
\begin{equation}\label{eq:bar-Phi}
    \lbar{\Phi}_t(t+j, t+i) := \left\{
    \begin{array}{l}
        (A-B_u\lbar{K}_{t+j-1|t})\cdots(A-B_u\lbar{K}_{t+i|t}) , \\
        \qquad \qquad\qquad\qquad\qquad 0\le i < j \le W \\
        I, \qquad \qquad \qquad \qquad \quad0\le i=j\le W
    \end{array}
    \right.
\end{equation}
$\lbar{K}_t, \lbar{K}_t^{d,i}$ are defined as follows:
\begin{equation}\label{eq:bar-K_t^i}
\begin{aligned}
    \lbar{K}_t&:=\lbar{K}_{t|t}\\
    \lbar{K}_t^{d,i}&:= (R_t+B_u^\top \lbar{P}_{t+1|t}B_u)^{-1}B_u^\top\lbar{\Phi}_t(i+1,t+1)^\top\lbar{P}_{i+1|t}
\end{aligned}
\end{equation}
Algorithm \ref{alg:online-alg} summarize the implementation of this MPC. 
\begin{algorithm}[htbp]
\begin{algorithmic}
\caption{Model Predictive Control (MPC) Algorithm} 
\label{alg:online-alg}
\REQUIRE $Q_{\min}, Q_{\max}, R_{\min}, R_{\max}, A, B_u, B_d$
\STATE Pre-calculate $P_{\max}$ using \eqref{eq:P_max}.
\FOR{$t = 1,2,\dots, T-1$}
\STATE Observe $x_t$ and receive predictions $\{d_{i|t}, Q_i, R_i\}_{i=t}^{t+W}$.
\IF{$t\le T-W-1$}
\STATE Calculate $\lbar{P}_{t+1|t},\dots,\lbar{P}_{t+W+1|t},\lbar{K}_{t|t},\dots,\lbar{K}_{t+W|t}$ using \eqref{eq:bar-K-P}.
\STATE Calculate $\lbar{K}_t, \lbar{K}_t^{d,i}$ using \eqref{eq:bar-Phi}\eqref{eq:bar-K_t^i}.
\STATE Implement $u_t$ as in \eqref{eq:MPC-rewrite}
\ELSE
\STATE Calculate $P_{t+1}, \dots, P_T, K_t, \dots, K_{T-1}$ directly via \eqref{eq:iterative-Riccati}\eqref{eq:K_t-LQR}.
\STATE Calculate $K_t^{d,i}$ using \eqref{eq:K_t^i}.
\STATE Implement $u_t$ as in \eqref{eq:optimal-policy}
\ENDIF
\ENDFOR
\end{algorithmic}
\end{algorithm}


Before going into the regret analysis in the next section, we first the exponential stability of the MPC algorithm despite the time-varying cost functions. 
\begin{prop}{(Exponential Stability of MPC)}\label{prop:stability-MPC}
Define the state transition matrix for MPC as:(Note that this is not the same matrix as $\lbar{\Phi}_t$ in \eqref{eq:bar-Phi}.)
\begin{equation*}
    \lbar{\Phi}^\textup{MPC}(t, t_0):= \left\{
    \begin{array}{cc}
         (A-B_u\lbar{K}_{t-1})\cdots(A-B_u\lbar{K}_{t_0}),   &  t>t_0\\
         I,& t= t_0 
    \end{array}
    \right.
\end{equation*}
Then:
$$\|\lbar{\Phi}^\textup{MPC}(t, t_0)\| \le \tau\rho^{t-t_0},$$ with $\tau,\rho$ defined the same as in Proposition \ref{prop:stability-LQR}
\end{prop}
Stability results of MPC is well studied in literature \cite{paolo08,diehl2010,angeli2011average,angeli2016,grune2014,grune2020} and this proposition almost follows the results in \cite{paolo08} exactly. Nevertheless, a self-contained proof is provided in our online report \cite{online_report} for reader's reference.

\section{Dynamic Regret Analysis}
Recall the definition of $\rho= \sqrt{1-\frac{\lambda_{\min}(Q_{\min})}{\lambda_{\max}(P_{\max})}}$ in Proposition~\ref{prop:stability-LQR} and further define another factor:
\begin{align*}
    \gamma &= \frac{\lambda_{\max}(A^\top P_{\max} A)}{\lambda_{\min}(Q_{\min}) + \lambda_{\max}(A^\top P_{\max} A)},
\end{align*}
The regret of the MPC in Algorithm~\ref{alg:online-alg} is bounded as follows.
\begin{theorem}\label{thm:regret-analysis-main}
The regret of MPC in  \eqref{eq:MPC-rewrite} is upper bounded by:
\begin{small}
\begin{equation}\label{eq:regret}
\begin{split}
   &  \textup{Regret}(\textup{MPC})\\
   &\le\! \alpha_1\!\underbrace{\left[\frac{1}{1\!-\rho}(\gamma^W\!+\rho^W) \!+ \gamma\frac{\rho{\!^W}\!-\gamma^{\!W}}{\rho\!-\gamma}\right]^2\left(\!\|x_1\|^2 \!+ \sum_{t=1}^T\|B_d d_t\|^2\!\right)}_{\textup{Part I}}\\
   &+ \!\alpha_2 \underbrace{\left(\frac{1}{1\!-\rho}\!+\frac{\gamma^W}{1-\rho^{2}}\right)\left(1\!+\frac{\gamma^{\!W}}{1\!-\rho}\right)\sum_{j=1}^{T-1}\sum_{i=j}^{j+W} \rho^{\!i\!-j}\|B_de_{i|j}\|^2}_{\textup{Part II}}),
\end{split}
\end{equation}
\end{small}
where $\alpha_1, \alpha_2$ are constants that only relate to $R_{\min}, R_{\max}, Q_{\min}, Q_{\max}, A, B_u$, but not $W$. 
\end{theorem}
The regret bound \eqref{eq:regret} consists of two terms. Part I depends on the total magnitude of the disturbances: $E_d:=\sum_{t=1}^T\|B_d d_t\|^2.$ Part II depends on the error of the predictions $\|B_de_{i|j}\|$. Before showing the proof, we make the following discussions on interpreting the regrets.

\vspace{4pt}
\noindent\textbf{Factors $\rho$ and $\gamma$:}
Proposition \ref{prop:stability-LQR} and \ref{prop:stability-MPC} show that the factor $\rho$ can be interpreted as the contraction factor of $\Phi, \lbar{\Phi}^\textup{MPC}$. We now further explain the meaning of $\gamma$. In a word, $\gamma$ captures the contraction of Ricatti iteration $F_{Q,R}(\cdot)$. In \cite{Krauth19}, they introduce a special metric on positive definite matrices
\begin{equation*}
    \delta_{\infty}(P,\lbar{P}) := \|\log(P^{-\frac{1}{2}}\lbar{P}P^{-\frac{1}{2}})\|,
\end{equation*}
 and show that $F_{Q,R}(\cdot)$ is contractive under $\delta_{\infty}(\cdot,\cdot)$, i.e.,
 \begin{equation}\label{eq:ricatti-contraction}
     \delta_{\infty}(F_{Q,R}(P), F_{Q,R}(\lbar{P})) \le \gamma \delta_{\infty}(P, \lbar{P}).
 \end{equation}
\eqref{eq:ricatti-contraction} is going to play a key role in bounding the difference $\|K_t - \lbar{K}_t\|$ and $\|K_t^{d,i} - \lbar{K}_t^{d,i}\|$ and therefore the cost difference between $J^\textup{MPC}$ and $J^*$, which will be further explained later. 

\vspace{4pt}
\noindent\textbf{Impact of $e_{i|j}$:} Let's take a more careful look at Part II, 
\begin{align*}
   \text{Part II} \sim~ O( \sum_{j=1}^{T-1}\sum_{i=j}^{j+W} \rho^{i-j}\|B_de_{i|j}\|^2)
\end{align*}
The coefficient in front of $\|B_de_{i|j}\|$ is dominated by $\rho^{i-j}$, which decays exponentially with $i-j$, which means that the effect of prediction inaccuracy decays exponentially fast as forecast goes to the far future. Since longer-term prediction tends to be less accurate, this suggests that MPC effectively alleviates the impact of multi-step-ahead prediction errors.

\vspace{4pt}
\noindent\textbf{Choice of $W$:} Part I involves the magnitude of the disturbances, and the coefficient in front of $\left(\|x_1\|^2 \!+ \sum_{t=1}^T\|B_d d_t\|^2\right)$
is roughly of scale:
\begin{align*}
   \text{Part I} \sim ~ O\left(\gamma_0^{2W} \left(\|x_1\|^2 \!+ \sum_{t=1}^T\|B_d d_t\|^2\right)\right)
\end{align*}
where $\gamma_0 = \max\{\rho, \gamma\}$. Thus if the prediction is accurate, i.e., $e_{i|j} = 0$, the regret decays exponentially fast with respect to $W$. This result matches with the conclusions in \cite{li2019}\cite{yu2020}. Additionally, when there's prediction error, the optimal choice of $W$ depends on the tradeoff between $E_d$ and the prediction errors.

For more insightful discussions, we consider non-decreasing $k-$step-ahead prediction errors, i.e. $\|B_de_{i+1|j}\|\ge \|B_de_{i|j}\|$. It can be shown that Part I increases with $E_d$ and Part II increases with prediction errors. Further, as $W$ increases, Part I decreases but Part II increases. Thus when Part I dominates the regret bound, i.e. $E_d$ is large when compared with the prediction errors, selecting a large $W$ reduces the regret bound. On the contrary, when Part II dominates the regret bound, i.e. the prediction errors are large when compared with $E_d$, a small $W$ is preferred. The choice of $W$ above are quite intuitive: when the environment is perturbed by large disturbances but the disturbances could be roughly predicted, one should use more predictions to prepare for future changes; however, with poor prediction and mildly perturbed environments, one should ignore these long-term predictions whose quality tends to be low. Our the regret upper bound provides a quantitative way to pick a balanced prediction window $W$.

\subsection{Proof Sketch}
Due to the space limit, we only give a proof sketch of Theorem \ref{thm:regret-analysis-main} here. The detailed proof is available in our online report~\cite{online_report}. Roughly speaking the proof can be decomposed into 4 steps:
\begin{itemize}
    \item Applying `cost difference lemma' to derive a general regret formula, which decomposes the regret into the sum of terms w.r.t `control action error' $u_t^\textup{MPC} - u_t^*$
    \item Decomposing `control action error' $u_t^\textup{MPC} - u_t^*$ into i) `truncation error' $\sum_{t+W+1}^{T-1}K_t^{d,i}d_{i}$, ii) `prediction error' ($-\sum_{i=t}^{t+W}\lbar{K}_t^{d,i}e_{i|t}$), and iii)  `matrices approximation error'$(K_t - \lbar{K}_t)x_t^\textup{MPC} + \sum_{i=t}^{t+W}(K_t^i - \lbar{K}_t^{d,i})d_{i}$
    \item Bounding the `truncation error' and `prediction error' by bounding $\|K_t^{d,i}\|, i\ge t+W+1$ and $\lbar{K}_t^{d,i}, t\le i\le t+W$ respectively. 
    \item Bounding the `matrices approximation error' by bounding $\|K_t - \lbar{K}_t\|$ and $\|K_t^{d,i} - \lbar{K}_t^{d,i}\|$, along with the exponential stability of online MPC algorithm.
\end{itemize}

\vspace{4pt}
\noindent\textit{\textbf{Step 1:}} We first introduce our general regret bound derived from `cost difference lemma'. Cost difference lemma is a very well-know lemma in reinforcement learning (RL) community and plays an important role in deriving many RL algorithms such as Trust Region Policy Optimization (TRPO \cite{schulman2015}). In \cite{fazel2018}, they also use it to prove gradient dominance for standard, time-invariant LQR problem. However, previous results are mostly for time-invariant systems and we haven't found an exact formula that applies the lemma to LQR problems with time-varying cost functions and disturbances. Thus we now state our result:
\begin{lemma}{(General Regret Formula)}
\label{lemma:regret-formula}
For an online control algorithm $\bm\pi$ that satisfies:
\begin{equation*}
    u_t = \pi_t(x_t, \{d_\tau, Q_\tau, R_\tau\}_{\tau=1}^t, \{d_{i|t,Q_i,R_i}\}_{i=t}^{t+W}),
\end{equation*}
i.e., control action does not rely on history $\{x_\tau\}_{\tau=1}^{t-1}$, the regret can be written as:
\begin{equation}
\begin{split}
  \textup{Regret}(\bm \pi)= \sum_{t=1}^{T-1} (u_t^{\bm \pi} - u_t^*)^\top (R_t + B_u^\top P_{t+1} B_u) (u_t^{\bm \pi} - u_t^*) 
\end{split}
\end{equation}
\begin{align*}
    \text{where: }\quad u_t^{\bm \pi} &= \pi_t(x_t^{\bm\pi}, \{d_\tau, Q_\tau, R_\tau\}_{\tau=1}^t, \{d_{i|t,Q_i,R_i}\}_{i=t}^{t+W}),\\
    u_t^* &= \pi_t^*(x_t^{\bm \pi}, d_t,\dots,d_{T-1})
\end{align*}
and that $\{x_t^{\bm \pi}\}$ is the trajectory generated by the online control algorithm  $\bm \pi$.
\end{lemma}
The proof of Lemma \ref{lemma:regret-formula} is given in the Appendix. Lemma \ref{lemma:regret-formula} successfully decouples the regret into the summation of a quadratic cost on `control action error': $u_t^{\bm \pi} - u_t^*$, so instead of analyzing $\textup{Regret}(\bm \pi)$ that entangles control action in all time step, we only need to bound the `control action error' at each time step separately. Note Lemma \ref{lemma:regret-formula} holds for policy $\bm \pi$ that can be more general than MPC algorithm.

\vspace{4pt}
\noindent\textit{\textbf{Step 2:}} From \eqref{eq:optimal-policy} \eqref{eq:MPC-rewrite}, the `control action error' of MPC algorithm can be decomposed as:
\begin{equation} \label{eq:control_difference}
\begin{split}
      u_t^\textup{MPC} - u_t^* =  \underbrace{\sum_{i=t+W+1}^{T-1}K_t^{d,i}d_{i}}_{\text{Truncation Error}} -\underbrace{\sum_{i=t}^{t+W}\lbar{K}_t^{d,i}e_{i|t}}_{\text{Prediction Error}} \\+ \underbrace{(K_t - \lbar{K}_t)x_t^\textup{MPC} + \sum_{i=t}^{t+W}(K_t^{d,i} - \lbar{K}_t^{d,i})d_{i}}_{\text{Matrices Approximation Error}} .
\end{split}
\end{equation}
We name the term $\sum_{t+W+1}^{T-1}K_t^{d,i}d_{i}$ as `truncation error', since this term appears when we throw away future disturbances after step $t+W$. The term $(K_t - \lbar{K}_t)x_t^\textup{MPC} + \sum_{i=t}^{t+W}(K_t^{d,i} - \lbar{K}_t^{d,i})d_{i}$ is named `matrices approximation error' because this term will disappear if the controller has knowledge of all future $Q_t, R_t$'s. Note that in \cite{yu2020}'s setting, they study constant $Q_t = Q, R_t = R$, in which case matrices approximation error is zero and need not be taken into consideration. The term $-\sum_{i=t}^{t+W}\lbar{K}_t^{d,i}e_{i|t}$ is called `prediction error' for the reason that it is related to the prediction accuracy at time step $t$.

\vspace{4pt}
\noindent\textit{\textbf{Step 3:}} We'll look into truncation error and prediction error first, because they are easier to analyse. It has been pointed out in Corollary \ref{coro:bound k_t^i} that $\|K_t^{d,i}\|$ decays exponentially w.r.t. $i-t$. Note that for truncation error $\sum_{t+W+1}^{T-1}K_t^{d,i}d_{i}$ the summation start from $t+W+1$, thus roughly speaking this term is of order $O(\rho^{W+1})$.

As for prediction error, similar to Proposition \ref{prop:stability-LQR} and Corollary \ref{coro:bound k_t^i}, we can show that
\begin{equation*}
     \|\lbar{\Phi}_t(t+j,t+i)\| \le \tau \rho^{j-i},~ j\ge i
\end{equation*}
\begin{equation*}
    \|\lbar{K}_t^{d,i}\| \le \frac{\tau\|B_u\|\lambda_{\max}(P_{\max})}{\lambda_{\min}(R_{\min})}\rho^{i-t}, i\ge t
\end{equation*}
Thus $\|\lbar{K}_t^{d,i}\|$ is roughly of scale $\rho^{i-t}$, which suggests that the effect caused by prediction error decays exponentially as prediction goes to the far future.

\vspace{4pt}
\noindent\textit{\textbf{Step 4: }}`Matrices approximation error' is a little bit harder to handle, because we not only need to bound the error term $\|K_t - \lbar{K}_t\|, \|K_t^{d,i} - \lbar{K}_t^{d,i}\|$, but also need to make sure that $\{x_t^\textup{MPC}\}$, which is the trajectory generated by the MPC algorithm, does not explode and become unbounded. The latter is guaranteed by the exponential stability of MPC algorithm, while bounding the error terms $\|K_t - \lbar{K}_t\|, \|K_t^{d,i} - \lbar{K}_t^{d,i}\|$ requires the contractivity of Ricatti iteration shown in \eqref{eq:ricatti-contraction}. This is the point where factor $\gamma$ emerges in the regret.
\begin{lemma}\label{lemma:bound-K_t^i-difference}
\begin{align*}
    &\|K_t - \lbar{K}_t\| \le \alpha_3 \gamma^W\\
    &\|K_t^{d,i}-\lbar{K}_t^{d,i}\|\le \alpha_4 \gamma^{W-i+t}\rho^{i-t}, ~t\le i\le t+W
\end{align*}
where,
\begin{align*}
    &\alpha_3 = 2\frac{\|A\|\|B_u\|}{\lambda_{\min}(R_{\min})}\frac{\lambda_{\max}(P_{\max})^4}{\lambda_{\min}(Q_{\min})^2}\left(\|B_uR_{\min}^{-1}B_u^\top\| + 1\right)\frac{1}{1-\gamma}\\
    &\alpha_4 = 2\frac{\|B_u\|}{\lambda_{\min}(R_{\min})}\frac{\lambda_{\max}(P_{\max})^4}{\lambda_{\min}(Q_{\min})^2}\left(\|B_uR_{\min}^{-1}B_u^\top\| + 1\right)\frac{1}{1-\gamma}
    \end{align*}
\end{lemma}

The key step in the proof of the Lemma \ref{lemma:bound-K_t^i-difference} is to bound $\|\lbar{P}_{i+1|t} - P_{i+1}\|$ which is provided below
\begin{lemma}\label{lemma:bound-P_t-difference}
\begin{equation*}
   \|\lbar{P}_{i+1|t} - P_{i+1}\| \le \gamma^{t+W-i}~\frac{\lambda_{\max}(P_{\max})^2}{\lambda_{\min}(Q_{\min})}, \quad  t\le i\le t+W.
\end{equation*}
\end{lemma}
Lemma \ref{lemma:bound-P_t-difference} is a direct corollary from \eqref{eq:ricatti-contraction}. Note that
\begin{equation*}
    P_{i} = F_{Q_{i}, R_{i}}(P_{i+1}), ~
    \lbar{P}_{i|t} = F_{Q_{i},R_{i}}(\lbar{P}_{i+1|t})
\end{equation*}
Applying \eqref{eq:ricatti-contraction} immediately gives us:
\begin{align*}
    \delta_{\infty}(\lbar{P}_{i|t}, P_{i}) &\le \gamma \delta_{\infty}(\lbar{P}_{i+1|t}, P_{i+1})
    \le \cdots \\&\le \gamma^{t+W+1-i} \delta_{\infty}(\lbar{P}_{t+W+1|t}, P_{t+W+1}) 
\end{align*}
which could be used to derive Lemma~\ref{lemma:bound-P_t-difference}. 


Combining all 4 steps together will give the upper bound in Theorem \ref{thm:regret-analysis-main}. The rigorous  proof and computation are quite cumbersome and we refer readers to the Appendix of our online report \cite{online_report} for more details.
\section{Numerical Simulations}
To numerically test 
\begin{align*}
    A = \left[
    \begin{array}{cc}
        0 & 1 \\
        1 & 0
    \end{array}
    \right], ~B_u = \left[
    \begin{array}{c}
    0\\
    1
    \end{array}
    \right], ~B_d = \left[
    \begin{array}{c}
    0\\
    1
    \end{array}
    \right]
\end{align*}
The stage cost function $Q_t, R_t$ are randomly picked as
\begin{equation*}
    Q_t = q_t I, ~ R_t = r_t
\end{equation*}
where $q_t$'s are picked randomly from Unif[2,3], and $r_t$'s are picked from Unif[5,6]. Disturbances $d_t$ are drawn randomly from standard Gaussian distribution, $d_t \sim \mathcal{N}(0,1)$. We consider the following two settings:

\vspace{4pt}
\noindent\textbf{Accurate Prediction on $d_t$:} The numerical result for accurate prediction case is shown in Figure \ref{fig:numerics} (Left). It displays the relationship of dynamical regret and prediction window size $W$. According to \eqref{eq:regret}, when setting prediction errors $e_{i|j} = 0$, the regret should be exponentially decreasing w.r.t. $W$, which matches the simulation result.

\vspace{4pt}
\noindent\textbf{Noisy Prediction on $d_t$:} The numerical result is shown in Figure \ref{fig:numerics} (Right). For this setting, we set the prediction errors as $e_{i|j} \sim \text{snr}*N(0,1)$, where snr denotes the signal-to-noise ratio of the predictions. Smaller snr indicates more accurate predictions. The figure suggests that for more accurate predictions, choosing a fairly large prediction window will boost the performance of the controller; while for large snr values, a longer prediction does not benefit too much, because the prediction error term is the dominant term of the regret.

\begin{figure}
\includegraphics[width = .5\linewidth]{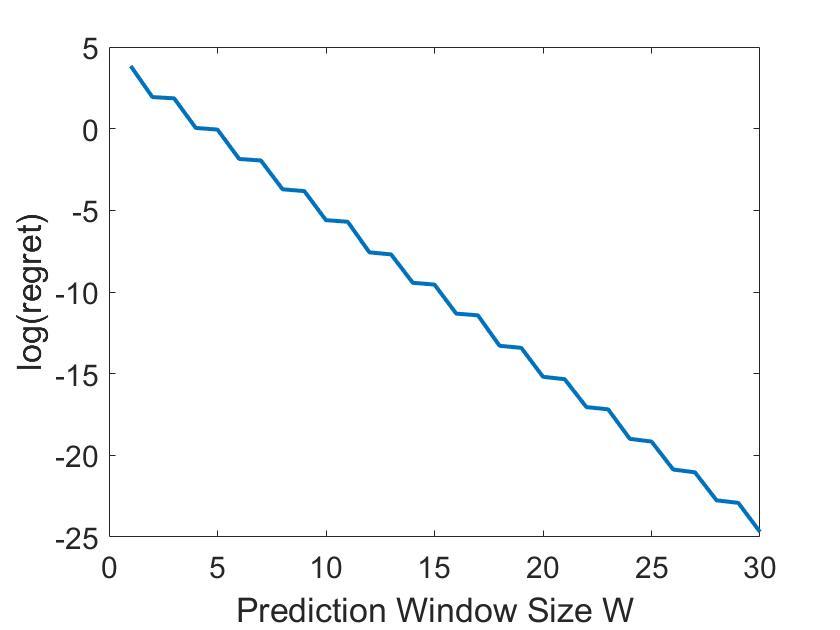}\includegraphics[width = .5\linewidth]{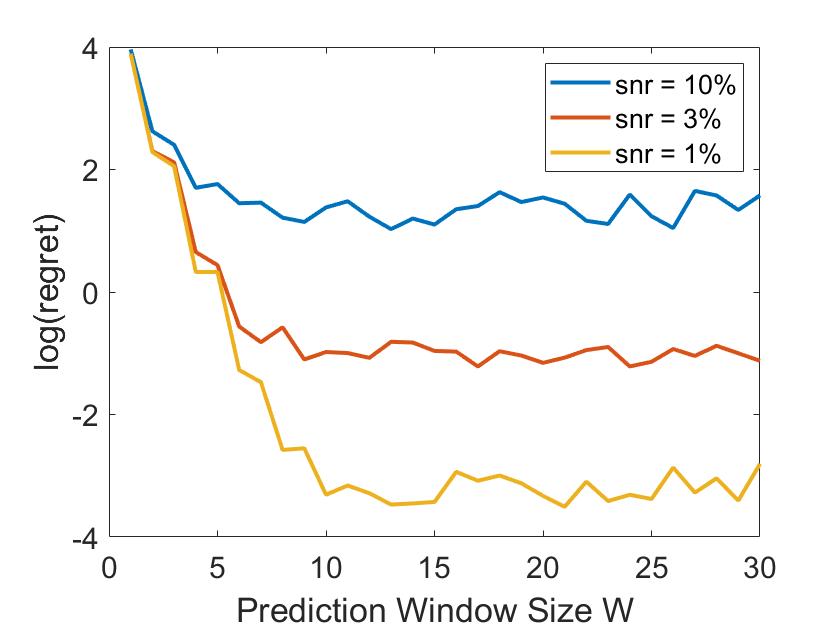}
\caption{(Left): Regret for accurate prediction on $d_t$; (Right): Regret for accurate prediction on $d_t$, 'snr' represents the signal-to-noise ratio of the predictions.}
\label{fig:numerics}
\end{figure}

\section{Conclusion}
This paper studies the role of predictions on dynamic regrets of online linear quadratic regulator with time-varying cost functions and disturbances. Besides giving an explicit dynamic regret upper bound for model predictive control algorithm in this setting, we also proposed a regret analysis framework based on `cost difference lemma' that might be applicable for more general class of control algorithms than MPC. This paper leads to many interesting future directions, some of which are briefly discussed below. The first direction is to further allow noisy predictions on $Q_t, R_t$. The second direction is to consider time varying system parameters $A_t, B_{u_t}, B_{d_t}$, e.g. \cite{grune2017,grune2020}. Moreover, it might be interesting to consider the setting where system parameters are unknown, possibly by applying learning based control tools, e.g., \cite{dean2018regret,Krauth19,ouyang2017learning}.

\bibliographystyle{IEEEtran}

\appendix


\subsection{Proof of Proposition \ref{prop:optimal-policy} and Lemma \ref{lemma:regret-formula}} \label{apdx:cost-difference-lemma}

We first define the optimal cost-to-go function:
\begin{equation*}
\begin{aligned}
    &\quad V_t^*(x):=\\
    &\min_{\{\!u_t,\dots,u_{T\!-1}\!\}} \left[\sum_{k=t}^{T-1}(x_k^\top Q_k x_k \!+\! u_k^\top R_ku_k) \!+\! x_T^\top Q_Tx_T ~|x_t = x\right]\\
    &\qquad\qquad\qquad s.t. \quad x_{k+1} = Ax_k + B_uu_k + B_dd_k
\end{aligned}
\end{equation*}
For simplicity we further define the following variables:
\begin{align*}
    \cA^{(n)}:= \left[
    \begin{array}{c}
         I  \\
         A  \\
         \vdots \\
         A^{n-1}
    \end{array}
    \right],~y_t^{t:T}\!(x) \!:=\! \left[\!
    \begin{array}{c}
         y_t^1 \\
         y_t^2  \\
         \vdots \\
         y_t^{T\!-t\!+1}
    \end{array}
    \!\right],
\end{align*}
where $y_t^{t:T}\!\in \mathbb{R}^{(T\!-t\!+1)n}$ ,and the recursive relationship between $y_t^i$ satisfies:
\begin{align*}
    y_t^1 &= x,\\
    y_t^{i+1} &= A y_t^i + B_dd_{t+i-1},~ i=1,2,\dots, T-t.
\end{align*}
Further define:
\vspace{-6pt}
\begin{equation*}
    y_t^{t+1:T}(x) : = \left[
    \begin{array}{c}
         y_t^2  \\
         \vdots \\
         y_t^{T-t+1}
    \end{array}
    \right].
\end{equation*}

By the definition of $\cA^{(n)}, ~ y_t^{t:T}, ~y_t^{t+1:T} $, we have that:
\begin{align*}
y_{t+1}^{t+1:T}(Ax+B_uu+B_dd_t) = y_t^{t+1:T}(x) + \cA^{(T-t)}B_u u.
\end{align*}
Additionally, by the definition of $V_t^*$, we have the following Bellman equation:
\begin{equation}\label{eq:Bellman-eq-with-disturbance}
\begin{split}
     V_t^*(x) = \min_u x^\top Q_tx \!+\! u^\top R_tu
     \!+\! V_{t+1}^*(Ax\!+\!B_uu\!+\!B_dd_t)
\end{split}
\end{equation}

The following proposition gives an explicit form for $V_t^*$.
\begin{prop}\label{prop:value-function-form}
$V_t^*(x)$ can be written as:
\begin{equation*}
\begin{split}
     & V_t^*(x) =\quad y_t^{t:T}(x)^\top V_t y_t^{t:T}(x),
\end{split}
\end{equation*}
where $V_t \in \mathbb{R}^{(T-t+1)n\times(T-t+1)n}$ is defined by the following recursive relationship:
\begin{small}
\begin{equation}\label{eq:Value-function-form}
\begin{split}
    &V_T = Q_T\\
    &V_t =\left[
    \begin{array}{cc}
        Q_t & 0 \\
        0 & f_{t}(V_{t+1})
    \end{array}
    \right],~ t = T-1, T-2, \dots, 1.
\end{split}
\end{equation}
\end{small}
where
\begin{small}
\begin{equation*}
\begin{split}
    &\quad f_{t}(X) = \\
    & X - X\cA^{(T\!-\!t)}B_u(R_t+B_u^\top\cA^{(T\!-\!t)^\top} X \cA^{(T\!-\!t)} B_u)^{\!-\!1}B_u^\top\cA^{(T\!-\!t)^\top} X.
\end{split}
\end{equation*}
\end{small}
The optimal policy is given by:
\begin{equation}\label{eq:optimal-controller}
    u_t^* = -G_t y_t^{t+1:T}(x),
\end{equation}
where 
\vspace{-6pt}
$$G_t:=(R_t\!+\!B_u^\top\cA^{(T\!-\!t)^\top} V_{t+1} \cA^{(T\!-\!t)} B_u)^{\!-\!1} B_u^\top\cA^{(T\!-\!t)^\top} V_{t+1}$$ 
\end{prop}
\begin{proof}
We use proof of induction. It is quite obvious that $V_T^*(x) = x^\top Q_Tx$.
Suppose proposition holds for time step $t+1$, i.e.
\begin{equation*}
\begin{split}
    V_{t+1}^*(x) = y_{t+1}^{t+1:T}(x)^\top V_{t+1} y_{t+1}^{t+1:T}(x).
\end{split}
\end{equation*}
Here for the sake of simplicity we write $y_t^{t:T}(x;d_t,\dots,d_{T-1})$ as $y_t^{t:T}(x)$.

By Bellman equation \eqref{eq:Bellman-eq-with-disturbance} we know that
\begin{align*}
    & V_t^*(x) \!=\! \min_u x^\top Q_tx \!+\! u^\top R_tu \!+\! V_{t+1}^*(Ax\!+\!B_uu\!+\!B_dd_t)\\
    &= \min_u x^\top Q_tx + u^\top R_tu \\&+ y_{t+1}^{t+1:T}(Ax\!+\!B_uu\!+\!B_dd_t)^\top V_{t+1}y_{t+1}^{t+1:T}(Ax\!+\!B_uu\!+\!B_dd_t)\\
    & = \min_u x^\top Q_tx + u^\top R_tu\\ &+
    (y_t^{t+1:T}(x) \!+\! \cA^{(T-t)}B_u u)^\top V_{t+1} (y_t^{t+1:T}(x) \!+\! \cA^{(T-t)}B_u u).
\end{align*}
Minimizing over $u$ we get that:
\begin{equation}\label{eq:u-star}
\begin{split}
        &\qquad u^* =\\ 
        &-(R_t\!+\!B_u^\top\!\cA^{(\!T\!-\!t\!)^\top} V_{t+1} \cA^{(\!T\!-\!t\!)} B_u)^{\!-\!1}
        B_u^\top\cA^{(\!T\!-\!t\!))^\top} V_{t+1} y_t^{t\!+\!1:\!T}(x)
\end{split}
\end{equation}

Substituting \eqref{eq:u-star} into the Bellman equation we get:
\begin{align*}
    V_t(x) &= x^\top Q_t x+y_t^{t+1:T}(x)^\top f_t(V_{t+1})y_t^{t+1:T}(x)\\
    &= y_t^{t:T}(x)^\top \left[
    \begin{array}{cc}
        Q_t & 0 \\
        0 & f_t(V_{t+1})
    \end{array}
    \right] y_t^{t+1:T},
\end{align*}
which completes the proof.
\end{proof}

The next proposition shows that Proposition \ref{prop:value-function-form} is consistent with classical LQR results.

\begin{prop}\label{prop:consistency-with-LQR}{(Consistency with Standard LQR)}
Define $P_t \in \mathbb{R}^{n\times n}$ as:
\vspace{-6pt}
\begin{equation}\label{eq:defi-tilde-P}
P_{t} := \cA^{(T-t+1)^\top} V_t \cA^{(T-t+1)}.
\end{equation}
Then we have:
\begin{small}
\begin{equation*}
    P_t = Q_t + A^\top P_{t+1} A - A^\top P_{t+1} B_u (R_t + B_u^\top P_{t+1}B_u)^{-1}B_u^\top P_{t+1}A,
\end{equation*}
\end{small}
which is the same matrix as the value function matrix for LQR defined in \eqref{eq:iterative-Riccati}.
\end{prop}
We omit the proof for Proposition \ref{prop:consistency-with-LQR}, since it can be done by simple algebraic manipulation.

We are now ready to prove Proposition \ref{prop:optimal-policy}. Though there are some very similar results in literature, we'll give our own proof for the optimal offline controller. Note that we have already shown in Proposition \ref{prop:value-function-form} that the optimal policy is given by \eqref{eq:optimal-controller}, thus we only need to show that \eqref{eq:optimal-policy} and \eqref{eq:optimal-controller} are equivalent.

The proof of Proposition \ref{prop:optimal-policy} depends on the following lemma:
\begin{lemma}\label{prop:Y-form}
Define $Y_t\in \mathbb{R}^{(T-t+1)n\times n}$ as $Y_T := V_t \cA^{(T-t+1)}$, then $Y_t$ has an explicit form:
\begin{equation*}
    Y_t = \left[
    \begin{array}{c}
         Q_t  \\
         Q_{t+1}(A-B_uK_t)\\
         Q_{t+2}(A-B_uK_{t+1})(A-B_uK_{t})\\
         \vdots\\
         Q_T(A-B_uK_{T-1})\cdots(A-B_uK_t)
    \end{array}
    \right],
\end{equation*}
where $K_t$ is the optimal control gain for standard LQR problem as defined in \eqref{eq:K_t-LQR}.
\end{lemma}
\begin{proof}
From eq\eqref{eq:Value-function-form} and Proposition \ref{prop:consistency-with-LQR}, with some simple algebraic computation we can derive the recursive relationship between $Y_t$:
\begin{equation}\label{eq:relationship-Y}
\begin{split}
    Y_t &= \left[
    \begin{array}{c}
         Q_t  \\
         Y_{t+1}A-Y_{t+1}B_u(R_t+B_u^\top P_{t+1}B_u)^{-1}B_u^\top P_{t+1}A 
    \end{array}
    \right]\\ &= \left[
    \begin{array}{c}
         Q_t  \\
         Y_{t+1}(A - B_uK_t)
    \end{array}
    \right]
\end{split}
\end{equation}
Applying this relationship recursively we can get that:
\begin{equation*}
    Y_t = \left[
    \begin{array}{c}
         Q_t  \\
         Q_{t+1}(A-B_uK_t)\\
         Q_{t+2}(A-B_uK_{t+1})(A-B_uK_{t})\\
         \vdots\\
         Q_T(A-B_uK_{T-1})\cdots(A-B_uK_t)
    \end{array}
    \right],
\end{equation*}
which completes the proof.
\end{proof}

\begin{proof}{(Proposition \ref{prop:optimal-policy})}
According to \eqref{eq:optimal-controller},
\begin{align*}
    u_t &= -G_t y_t^{t+1:T}(x_t)\\
    &=-(R_t+B_u^\top P_{t+1} B_u)^{-1}B_u^\top\cA^{(T-t)^\top} V_{t+1}y_t^{t+1:T}(x_t)\\
    & = -(R_t+B_u^\top P_{t+1} B_u)^{-1}B_u^\top Y_{t+1}^\top y_t^{t+1:T}(x_t)\\
    & =  -(R_t+B_u^\top P_{t+1} B_u)^{-1}B_u^\top Y_{t+1}^\top \left(\cA^{(T-t)}Ax_t\right.\\& + \cA^{(T-t)}B_dd_t + \left[
    \begin{array}{c}
         0  \\
         \cA^{(T-t-1)} 
    \end{array}
    \right]B_dd_{t+1}\\
    &\qquad\left.
    +~\cdots~+\left[
    \begin{array}{c}
         0  \\
         \vdots\\
         0\\
         I
    \end{array}
    \right]B_d d_{T-1}
    \right)
\end{align*}
Since $Y_{t+1} = V_{t+1}\cA^{(T-t)}$, we have that:
$$Y_{t+1}^\top\cA^{(T-t)} = P_{t+1}$$
Additionally, from \eqref{eq:relationship-Y} we have:
\begin{align*}
    &Y_{t+1}^\top\left[
    \begin{array}{c}
         0  \\
         \cA^{(T-t-1)} 
    \end{array}
    \right] \\=& \left[Q_{t+1},  (A-B_u K_{t+1})^\top Y_{t+2}^\top\right]\left[
    \begin{array}{c}
         0  \\
         \cA^{(T-t-1)} 
    \end{array}
    \right] \\=& (A - B_uK_{t+1})^\top P_{t+2}
\end{align*}
Similarly,
\begin{equation*}
\begin{split}
   &\quad Y_{t+1}^\top\left[
    \begin{array}{c}
         0  \\
         \vdots\\
         0\\
         \cA^{(T-t-k)}
    \end{array}
    \right]\\& = (A-B_uK_{t+1})^\top\cdots(A-B_uK_{t+k})^\top P_{t+k+1}\\
    &= \Phi(t+k+1,t+1)^\top P_{t+k+1}
    \end{split}
\end{equation*}
Thus we have that:
\begin{equation*}
\begin{split}
    u_t & =  -(R_t+B_u^\top P_{t+1} B_u)^{-1}B_u^\top Y_{t+1}^\top \left(\cA^{(T-t)}Ax_t\right. \\&+ \cA^{(T-t)}B_dd_t + \left[
    \begin{array}{c}
         0  \\
         \cA^{T-t-1} 
    \end{array}
    \right]B_dd_{t+1}\\
    &\qquad\left.
    +~\cdots~+\left[
    \begin{array}{c}
         0  \\
         \vdots\\
         0\\
         I
    \end{array}
    \right]B_d d_{T-1}
    \right)\\
    &= -(R_t+B_u^\top P_{t+1} B_u)^{-1}B_u^\top P_{t+1} Ax_t\\&  -(R_t+B_u^\top P_{t+1} B_u)^{-1}B_u^\top P_{t+1} B_dd_t- \cdots\\ & 
    - (R_t+B_u^\top P_{t+1} B_u)^{-1}B_u^\top \Phi(T,t+1)^\top P_{T}B_dd_{T-1}\\
    &= -K_tx_t - K_t^{d,t}B_dd_t - \cdots - K_t^{d,T-1}B_dd_{T-1}
\end{split}
\end{equation*}
which completes the proof.
\end{proof}

Define the value function for a given policy $\bm \pi$:
\begin{equation*}
\begin{split}
     V_t^{\bm \pi}(x):= \left[\sum_{k=t}^{T-1}(x_k^\top Q_kx_k + u_k^\top R_ku_k) + x_T^\top Q_T x_T\right.\\
     \left.|~x_t = x,\{u_k=\bm \pi_k(I_k)\}_{k=t}^{T-1}\right],
\end{split}
\end{equation*}
where $I_k$ denotes the information available for the controller at time step $k$. Specifically for MPC considered in this paper, $I_k = \left\{x_k, \{d_{i|t}, Q_i, R_i\}_{i=k}^{k+W}\right\}$. 

Similarly we could also define the $Q$-function:
\begin{equation*}
\begin{split}
        \cQ_t^{\bm \pi}(x,u) := \left[\sum_{k=t}^{T-1}(x_k^\top Q_kx_k + u_k^\top R_ku_k) + x_T^\top Q_T x_T\right.\\
        |~x_t = x,\left.u_t = u,\{u_k= \pi_k(I_k)\}_{k=t}^{T-1}\right]
\end{split}
\end{equation*}
We  state a  general version of cost difference lemma.
\begin{lemma}{(Cost Difference Lemma)}\label{lemma:cost-difference}
For two  policies $\bm \pi_1, \bm \pi_2$, the difference of their regrets can be represented by:
\begin{equation}
 V_1^{\bm \pi_2}(x) - V_1^{\bm \pi_1}(x) = \sum_{t=1}^{T-1} \cQ_t^{\bm \pi_1}(x_t^{\bm \pi_2},u_t^{\bm \pi_2}) - V_t^{\bm \pi_1}(x_t^{\bm \pi_2})
\end{equation}
where $\{x_t^{\bm \pi_2}, u_t^{\bm \pi_2}\}$ are trajectory generated by starting at initial state $x$ and imposing policy $\bm \pi_2$.
\end{lemma}
The proof of Lemma \ref{lemma:cost-difference} can be found in literature \cite{fazel2018,schulman2015}.
By applying Lemma \ref{lemma:cost-difference}, we get the proof of Lemma \ref{lemma:regret-formula}.


\noindent\textit{Proof: }{(Lemma \ref{lemma:regret-formula})}
\vspace{-8pt}
\begin{align*}
& \textup{Regret}(\bm \pi)= V_1^{\bm \pi}(x) - V_1^{*}(x)=  \sum_{t=1}^{T-1} \cQ_t^{*}(x_t^{\bm \pi},u_t^{\bm \pi}) - V_t^{*}(x_t^{\bm \pi})\\
    & = \sum_{t=1}^{T-1} \left[x_t^{\bm \pi}{^\top} Q_tx_t^{\bm \pi} \!+ \!u_t^{\bm \pi}{^\top} R_tu_t^{\bm \pi} \!+\! V_{t+1}^*(Ax_t^{\bm \pi}\!+\!B_uu_t^{\bm \pi}\!+\!B_dd_t)\right.\\
    & -x_t^{\bm \pi}{^\top} Q_tx_t^{\bm \pi} - u_t^{*\top} R_tu_t^* \left.- V_{t+1}^*(Ax_t^{\bm \pi}\!+\!B_uu_t^*\!+\!B_dd_t)\right]\\
    &=\sum_{t=1}^{T-1} \left[u_t^{\bm \pi}{^\top} R_tu_t^{\bm \pi} +\left(y_t^{t+1:T}(x_t^{\bm \pi}) + \cA^{(T-t)}B_uu_t^{\bm \pi}\right)^\top\right.\\ &\qquad\qquad\qquad\qquad\cdot~ V_{t+1}\left(y_t^{t+1:T}(x_t^{\bm \pi}) + \cA^{(T-t)}B_uu_t^{\bm \pi}\right)\\
    &\qquad - u_t^{*\top} R_tu_t^* - \left(y_t^{t+1:T}(x_t^{\bm \pi}) + \cA^{(T-t)}B_uu_t^*\right)^\top\\& \left.\qquad\qquad\qquad\qquad \cdot ~V_{t+1}\left(y_t^{t+1:T}(x_t) + \cA^{(T-t)}B_uu_t^*\right) \right]\\
    &=\!\sum_{t=1}^{T-1}\!\left[(u_t^{\!\bm \pi}\!-\!u_t^{\!*}){\!^\top}\!(R_t\!+\!B_{\!u}^\top \cA^{(T\!-\!t)^{\!\top}}\!V_{t+1}\cA^{(T\!-\!t)}B_{\!u})(u_t^{\bm \pi}\!-\!u_t^*)\right.\\
    &\quad + 2(u_t^{\bm \pi}-u_t^*)^\top\left(R_tu_t^*+B_u^\top \cA^{(T-t)^\top}V_{t+1}\cA^{(T-t)}B_uu_t^*\right.\\&\left.\left.\qquad\qquad\qquad\qquad\qquad +B_u^\top \cA^{(T-t)^\top}V_{t+1}y_t^{t+1:T}(x_t^{\bm \pi})\right)\right]\\
    &=\!\sum_{t=1}^{T-1}\!\left[(u_t^{\bm \pi}-u_t^*)^{\!\top}(R_t+B_u^\top P_{t+1}B_u)(u_t^{\bm \pi}-u_t^*)\right.\\&\left.+ 2(u_t^{\bm \pi}-u_t^*)(R_t+B_u^\top P_{t+1}B_u)(u_t^* + G_ty_t^{t+1:T}(x_t^{\bm \pi}))\right]\\
    &=\sum_{t=1}^{T-1}(u_t^{\bm \pi}-u_t^*)^\top(R_t+B_u^\top P_{t+1}B_u)(u_t^{\bm \pi}-u_t^*) 
\end{align*}
\subsection{Proof of Theorem \ref{thm:regret-analysis-main}}\label{apdx:main-theorem}
In this section, we are going to give a rigorous proof for our main result Theorem \ref{thm:regret-analysis-main}.Proof of sketch in the main text has already given a clear structure how it is proved, what is left is purely tedious calculation and analysis. The major difficulty arises in the state trajectory generated by the MPC algorithm $x_t^\textup{MPC}$ showing up in \eqref{eq:control_difference}. We need to rewrite $x_t^\textup{MPC}$ in terms of initial state $x_1$, disturbances $ d_t$ and prediction errors $e{i|j}$, which makes the analysis quite tedious, but the structure should be clear. Thus for readers who are only looking for an intuitive understanding of the proof, we would suggest to first skip some of the detailed calculation.

As stated in the previous paragraph, the first step towards the proof is to rewrite $x_t^\textup{MPC}$ in terms of the initial state, disturbances and prediction errors, which is stated in the following proposition.
\begin{prop}\label{prop:x_t^pi}
Let $\{x_t^{\textup{MPC}}\}$ be the trajectory generated by MPC algorithm defined as in Eq \eqref{eq:MPC-rewrite}, then:
\begin{equation}\label{eq:x_t^pi}
\begin{split}
    x_t^{\textup{MPC}} &= \lbar{\Phi}^{\textup{MPC}}(t,1)x_1^{\textup{MPC}} + \sum_{i=1}^{\min\{t+W, T-1\}} M_{i|t}B_dd_i \\&- \sum_{j=1}^{t-1}\sum_{i=j}^{\min\{T-1,j+W\}} \lbar{\Phi}^{\textup{MPC}}(t,j+1)B_u\lbar{K}_j^{d,i}e_{i|j}
\end{split}
\end{equation}
where:
\begin{equation*}
    M_{i|t} \!=\! \sum_{j=\max\{\!1,i\!-\!W\!\}}^{\min\{\!t\!-\!1,i\!\}} \!\lbar{\Phi}^{\textup{MPC}}(t,j\!+\!1)(-B_u\lbar{K}_j^{d,i} + \mathbf{1}\{i\!=\!j\}I)
\end{equation*}
Additionally,
\begin{align*}
    &\|M_{i|t}\|\le c_4\rho^{t-i-1}, ~ 1\le i\le t-1\\
    &\|M_{i|t}\|\le c_5\rho^{i-t+1}, ~ i\ge t
\end{align*}
where,
\begin{align*}
    &c_4 = \|B_uR_{\min}^{-1}B_u^\top\|\lambda_{\max}(P_{\max})\tau^2\frac{1}{1-\rho^2} + \tau\\
    &c_5 = \|B_uR_{\min}^{-1}B_u^\top\|\lambda_{\max}(P_{\max})\tau^2\frac{1}{1-\rho^2}
\end{align*}
\end{prop}
\begin{proof}
Recall from \eqref{eq:MPC-rewrite}~-~\eqref{eq:bar-K_t^i} the MPC scheme, $x_t^{\textup{MPC}}$ can be written as:
\begin{align*}
    x_t^{\textup{MPC}}&= Ax_{t-1}^{\textup{MPC}} + B_uu_{t-1}^{\textup{MPC}} + B_dd_{t-1}\\
    &=(A\!-\!B_u\lbar{K}_{t-1})x_{t-1}^{\textup{MPC}} - B_u\lbar{K}_{t-1}^{d,t-1}B_dd_{t-1|t-1} - \cdots\\& - B_u\lbar{K}_{t-1}^{d,t+W-1}B_dd_{t+W-1|t-1} + B_dd_{t-1}\\
    &=(A\!-\!B_u\lbar{K}_{t-1})x_{t-1}^{\textup{MPC}} + (-B_u\lbar{K}_{t-1}^{d,t-1}+I)B_dd_{t-1} \\&- B_u\lbar{K}_{t-1}^{d,t}B_dd_t -\cdots - B_u\lbar{K}_{t-1}^{d,t+W-1}B_dd_{t+W-1} \\
    &-\!B_u\!\lbar{K}_{t\!-\!1}^{d,t-1}\!B_de_{t\!-\!1|t\!-\!1} \!- \!\cdots\! -\! B_u\!\lbar{K}_{t-1}^{d,t\!+\!W\!-\!1}\!B_de_{t\!+\!W\!-\!1|t\!-\!1}
\end{align*}
Applying the above equation iteratively will get \eqref{eq:x_t^pi}.

We now look at the bound on $\|M_{i|t}\|$. For $1\le i\le t-1$,
\begin{align*}
  M_{i|t}= \sum_{j=\max\{1, i-W\}}^{i}\lbar{\Phi}^{\textup{MPC}}(t,j+1)(-B_u\lbar{K}_j^{d,i})\\
  +\lbar{\Phi}^{\textup{MPC}}(t-1,i).
\end{align*}

Proposition \ref{prop:stability-MPC} has already showed that $\|\lbar{\Phi}^{\textup{MPC}}(t,t_0)\| \le \tau\rho^{t-t_0}$. Thus:
\begin{align*}
   &\qquad \|M_{i|t}\|\\ 
   &\le\!\sum_{j\!=\!\max\{\!1, i\!-\!W\!\}}^{i}\!\tau\!\rho^{t\!-\!1\!-\!j}\|B_u\!R_{\min}^{-1}\!B_u^\top\|\lambda_{\max}(\!P_{\max}\!)\tau\rho^{i\!-\!j} \!+\! \tau \rho^{t\!-\!1\!-\!i}\\
    &= \|B_u\!R_{\min}^{-1}\!B_u^\top\|\lambda_{\max}(\!P_{\max}\!)\tau^2\!\rho^{t\!-\!i\!-\!1}\left(\!\sum_{j\!=\!\max\{\!1,i\!-\!W\!\}}^i(\rho^2)^{i\!-\!j}\!\right) \\
    &\quad+ \tau \rho^{t-i-1}\\
    &\le \left(\|B_uR_{\min}^{-1}B_u^\top\|\lambda_{\max}(P_{\max})\tau^2\frac{1}{1-\rho^2} + \tau\right) \rho^{t-i-1}\\
    &= c_4 \rho^{t-i-1}
\end{align*}

For $i\ge t$,
\begin{align*}
    M_{i|t} &= \sum_{j=\max\{1,i-W\}}^{t-1}\lbar{\Phi}^{\textup{MPC}}(t,j+1)(-B_u\lbar{K}_j^{d,i})\\
    &\le\! \sum_{j\!=\!\max\{\!1,i\!-\!W\!\}}^{t\!-\!1}\tau\!\rho^{t\!-\!1\!-\!j}\|B_u\!R_{\min}^{-1}\!B_u^\top\|\lambda_{\max}(\!P_{\max}\!)\tau\!\rho^{i\!-\!j}\\
    & = \|B_u\!R_{\min}^{-1}\!B_u^\top\|\lambda_{\max}(\!P_{\max}\!)\tau^2\!\rho^{i\!-\!t\!+\!1}\!\sum_{j=\max\{\!1,i\!-\!W\!\}}^{t-1}\!(\rho^2\!)^{t\!-\!1\!-\!j}\\
    &\le \|B_uR_{\min}^{-1}B_u^\top\|\lambda_{\max}(P_{\max})\tau^2\frac{1}{1-\rho^2}\rho^{i-t+1} \\&= c_5\rho^{i-t+1}
\end{align*}
\end{proof}

Now that $x_t^\textup{MPC}$ is re-written in terms of $x_1, d_t, e_{i|j}$, the next step is to write the `control action error'$u_t^{\textup{MPC}} - u_t$ in these terms as well.
\begin{prop}\label{prop:control-difference-prop}
\begin{equation*}
\begin{split}
    u_t^{\textup{MPC}} - u_t = N_{0|t}x_1 + \sum_{i=1}^{T-1}N_{i|t}B_dd_i + \sum_{j=1}^t \sum_{i=j}^{j+W} L_{(i,j)|t} B_de_{i|j}
\end{split}
\end{equation*}
where:
\begin{align*}
    N_{i|t} &\!=\! \left\{
    \begin{array}{l}
        (K_t - \lbar{K}_t)\lbar{\Phi}^{\textup{MPC}}(t,1) ,~\quad\qquad\qquad\qquad i\!=\!0  \\
        (K_t - \lbar{K}_t)M_{i|t} , \quad\qquad\qquad\qquad1\le i\le t-1\\
        (K_t - \lbar{K}_t)M_{i|t} + (K_t^{d,i} - \lbar{K}_t^{d,i}) , ~ t \le i\le t\!+\!W\\
        (K_t - \lbar{K}_t)M_{i|t} + K_t^{d,i} ,~\quad\qquad i\ge t+W+1
    \end{array}
    \right.\\
    L_{(i,j)|t}&\!=\! \left\{\!
    \begin{array}{l}
        -\!(K_t \!-\! \lbar{K}_t)\lbar{\Phi}^{\textup{MPC}}\!(t,j\!+\!1)B_u\lbar{K}_j^{d,i} , j\!\le\! t\!-\!1, i\!\ge\! j  \\
        -\lbar{K}_t^{d,i} ,\quad\qquad\qquad\qquad\qquad ~1\le i\le t\!-\!1,j\!=\!t
    \end{array}
    \right.
\end{align*}
\end{prop}
\begin{proof}
\begin{equation*}
\begin{split}
     & u_t^{\textup{MPC}} - u_t^* =\\ &(K_t - \lbar{K}_t)x_t^{\textup{MPC}} + \sum_{i=t}^{t+W}(K_t^{d,i} - \lbar{K}_t^{d,i})d_{i} \\&+ \sum_{t+W+1}^{T-1}K_t^{d,i}d_{i} -\sum_{i=t}^{t+W}\lbar{K}_t^{d,i}e_{i|t}.
\end{split}
\end{equation*}
\begin{align*}
   &(K_t - \lbar{K}_t)x_t + \sum_{i=t}^{t+W}(K_t^{d,i} - \lbar{K}_t^{d,i})d_{i}+ \sum_{t+W+1}^{T-1}K_t^{d,i}d_{i}\\
    &=(K_t - \lbar{K}_t)\lbar{\Phi}^{\textup{MPC}}(t,1)x_1 + \sum_{i=1}^{t+W} (K_t - \lbar{K}_t)M_{i|t}B_dd_i
    \\\quad &- \sum_{j=1}^{t-1}\sum_{i=j}^{j+W} (K_t - \lbar{K}_t)\lbar{\Phi}^{\textup{MPC}}(t,j+1)B_u\lbar{K}_j^{d,i}e_{i|j} \\&+ \sum_{i=t}^{t+W}(K_t^{d,i} - \lbar{K}_t^{d,i})B_dd_i+ \sum_{t+W+1}^{T-1}K_t^{d,i}d_{i}\\
    &= N_{0|t}x_1 + \sum_{i=1}^{T-1}N_{i|t}B_dd_i + \sum_{j=1}^{t-1}\sum_{i=j}^{j+W} L_{(i,j)|t}B_de_{i|j}\\
\end{align*}
Combining the above two equations completes the proof.
\end{proof}
The next proposition bounds the norms of the matrices appeared in Proposition \ref{prop:control-difference-prop}.
\begin{prop}\label{prop:bound-N_i|t}
\begin{align*}
    &\|N_{0|t}\| \le c\gamma^W\rho^{t-1}\\
    &\|N_{i|t}\| \le c\gamma^W\rho^{t-i-1}, ~ 1\le i\le t-1\\
    &\|N_{i|t}\| \le c\gamma^{W-i+t}\rho^{i-t}, ~t\le i\le t+W\\
    &\|N_{i|t}\| \le c\rho^{i-t}, ~t\ge t+W+1\\
    &\|L_{(i,j|t)}\|\le c\gamma^W\rho^{t-2j+i-1}, ~ j\le t-1\\
    &\|L_{(i,t|t)}\|\le c\rho^{i-t}
\end{align*}
where the constant $c$ can be expressed by 
$$R_{\min}, R_{\max}, Q_{\min}, Q_{\max}, A, B_u$$
\end{prop}
\begin{proof}
From Corollary \ref{coro:bound k_t^i}, Lemma \ref{lemma:bound-K_t^i-difference}, \ref{prop:x_t^pi} we have that:
\begin{align*}
    &\|K_t^{d,i}\|, \|\lbar{K}_t^{d,i}\|\le c_1\rho^{i-t}, ~i\ge t\\
    &\|K_t - \lbar{K}_t\| \le c_2\gamma^W\\
    &\|K_t^{d,i} - \lbar{K}_t^{d,i}\| \le c_3 \gamma^{W-i+t}\rho^{i-t}\\
     &\|M_{i|t}\|\le c_4\rho^{t-i-1}, ~ 1\le i\le t-1\\
    &\|M_{i|t}\|\le c_5\rho^{i-t+1}, ~ i\ge t
\end{align*}
We now bound $\|N_{i|t}\|$:
\begin{align*}
    \|N_{0|t}\| &\le \|K_t - \lbar{K}_t\|\|\lbar{\Phi}^{\textup{MPC}}(t,1)\|\\
    &\le c_2\gamma^W\tau\rho^{t-1} \le c\gamma^W\rho^{t-1}
\end{align*}
For $1\le i\le t-1$,
\begin{align*}
    \|N_{i|t}\| &\le \|K_t - \lbar{K}_t\|\|M_{i|t}\|\\
    &\le c_1\gamma^W c_4\rho^{t-i-1} \le c\gamma^W\rho^{t-i-1}
\end{align*}
For $t\le i\le t+W$,
\begin{align*}
    \|N_{i|t}\| &\le \|K_t - \lbar{K}_t\|\|M_{i|t}\| + \|K_t^{d,i} - \lbar{K}_t^{d,i}\|\\
    &\le c_2\gamma^Wc_5\rho^{i-t+1} + c_3 \gamma^{W-i+t}\rho^{i-t}\\
    &\le (\rho c_2c_5 + c_3)\gamma^{W-i+t}\rho^{i-t} \le c \gamma^{W-i+t}\rho^{i-t}
\end{align*}
For $i\ge t+W+1$
\begin{align*}
    \|N_{i|t}\|&\le\|(K_t - \lbar{K}_t)M_{i|t} \|+\|K_t^{d,i}\|\\
    &\le c_2\gamma^W c_5\rho^{i-t+1} + 2c_1\rho^{i-t}\\
    &\le c\rho^{i-t}
\end{align*}
For $j\le t-1, i\ge j$
\begin{align*}
    \|L_{(i,j)|t}\|&=\|(K_t - \lbar{K}_t)\lbar{\Phi}^{\textup{MPC}}(t,j+1)B_u\bar{K}_j^{i-j+1}\|\\
    &\le c_2\gamma^W\tau\rho^{t-j-1}\|B_u\|c_1\rho^{i-j}\le c\gamma^W\rho^{t-2j+i-1} 
\end{align*}
For $i\ge t$
\begin{equation*}
    \|L_{(i,t)|t}\| =\|\bar{K}_t^{i-t+1}\|\le c_1\rho^{i-t}
\end{equation*}
\end{proof}

We are now ready to give a rigours proof for Theorem \ref{thm:regret-analysis-main}. According to Lemma \ref{lemma:regret-formula}, the regret is a quadratic sum over the `control action error' $(u_t^\textup{MPC}-u_t^*)$. The most part of the proof below is to translate this quadratic sum over the 'control action error' to a quadratic sum over $x_1, d_t, e_{i|j}$, which serves as another main source of cumbersome calculation.
\begin{proof}{(Theorem \ref{thm:regret-analysis-main})}
Applying the regret formula in Lemma \ref{lemma:regret-formula}, we have that
\begin{align*}
&\textup{Regret}(\bm \pi) =  V_1^{\bm\pi}(x)-V_1^*(x)\\
&=\sum_{t=1}^{T-1}(u_t^\textup{MPC} - u_t^*)(R_t+B_u^\top P_{t+1}B_u)(u_t^\textup{MPC} - u_t^*)\\
&\le \|R_{\max}+B_u^\top P_{\max}B_u\|\sum_{t=1}^{T-1}\|u_t^\textup{MPC} - u_t^*\|^2
\end{align*}
\begin{align*}
   & \|u_t^\textup{MPC} - u_t^*\|^2\\&= \|N_{0|t}x_1 + \sum_{i=1}^{T-1}N_{i|t}B_dd_i + \sum_{j=1}^t \sum_{i=j}^{T-1} L_{(i,j)|t} B_de_{i|j}\|^2\\
    &\le 2\|N_{0|t}x_1 + \sum_{i=1}^{T-1}N_{i|t}B_dd_i|\|^2 + 2 \|\sum_{j=1}^t\sum_{i=j}^{T-1} L_{(i,j)|t} B_de_{i|j}\|^2
\end{align*}
We first bound the term:
\begin{align*}
   &\quad \sum_{t=1}^T \|N_{0|t}x_1 + \sum_{i=1}^{T-1}N_{i|t}B_dd_i|\|^2\\
    &\le \sum_{t=1}^T \left(\|N_{0|t}\|\|x_1\| + \sum_{i=1}^{T-1}\|N_{i|t}\|\|B_dd_i|\|\right)^2\\
    &\le c^2 \sum_{t=1}^T\left(a_0^t\|x_1\| + \sum_{i=1}^{T-1}a_i^t\|B_dd_i|\|\right)^2\\
\end{align*}
The last inequality holds for the reason that Proposition \ref{prop:bound-N_i|t} implies
$$\|N_{i|t}\| \le c\cdot a_i^t,$$
where $$a_i^t = \gamma^{W-\min\left\{\max\{0,i-t\},W\right\}}\rho^{\max\{i-t, t-i-1\}}.$$

Applying Lemma \ref{lemma:auxillary} in Appendix \ref{apdx:others}, we have
\begin{align*}\label{eq:proof-final-eq2}
    &\quad \sum_{t=1}^T \|N_{0|t}x_1 + \sum_{i=1}^{T-1}N_{i|t}B_dd_i|\|^2\\&\le
    c^2\cdot\max_i\left\{\sum_{t=1}^{T-1}a_i^t\left(\sum_{j=0}^{T-1}a_j^t\right)\right\} \left(\|x_1\|^2 + \sum_{t=1}^{T-1}\|B_dd_t\|^2\right)
\end{align*}
For fixed $t$,
\begin{align*}
    &\sum_{j=0}^{T-1}a_j^t = \sum_{j=0}^{t-1}a_j^t +\! \sum_{j=t}^{\min\{\!t\!+\!W\!-\!1,T\!-\!1\!\}}\!a_j^t + \!\sum_{j \!= \!\min\{\!t\!+\!W\!-\!1,T\!-\!1\}\!+\!1}^{T-1}\!a_j^t\\
    &\le \sum_{j=-\infty}^{t-1}\gamma^W\rho^{t-j-1} + \sum_{j=t}^{t+W-1}\gamma^{W-j+t}\rho^{j-t} + \sum_{j=t+W}^{+\infty}\rho^{j-t}\\
    &= \gamma^W \frac{1}{1-\rho} + \sum_{i=0}^{W-1}\rho^{W-i}\rho^i + \rho^W\frac{1}{1-\rho}\\
    &= \frac{1}{1-\rho}(\gamma^W+\rho^W) + \gamma\frac{\rho^W-\gamma^W}{\rho-\gamma}\\
    &\le\left(\frac{2}{1-\gamma_0} + W\right)\gamma_0^W
\end{align*}
For fixed $i$,
\begin{align*}
   &\qquad \sum_{t=1}^{T-1}a_i^t\\ &= \sum_{t=1}^{\max\{0,i-W+1\}-1}a_i^t + \sum_{t=\max\{0,i-W+1\}}^{i}a_i^t + \sum_{t = i+1}^{T-1}a_i^t\\
    &\le \sum_{t=-\infty}^{\max\{0,i-W+1\}-1}\rho^{i-t} + \sum_{t=\max\{0,i-W+1\}}^{i} \gamma^{W-i+t}\rho^{i-t}\\&\qquad + \sum_{t = i+1}^{+\infty} \gamma^W\rho^{t-i-1}\\
    &= \rho^W\frac{1}{1-\rho} + \sum_{i=0}^{W-1}\rho^{W-i}\rho^i + \gamma^W \frac{1}{1-\rho}\\
    &= \frac{1}{1-\rho}(\gamma^W+\rho^W) + \gamma\frac{\rho^W-\gamma^W}{\rho-\gamma}\\
    &\le\left(\frac{2}{1-\gamma_0} + W\right)\gamma_0^W
\end{align*}
Thus:
\begin{equation}\label{eq:proof-final-eq3}
\begin{split}
&\quad \max_i\left\{\sum_{t=1}^{T-1}a_i^t\left(\sum_{j=0}^{T-1}a_j^t\right)\right\} \\&\le \left(\frac{1}{1-\rho}(\gamma^W+\rho^W) + \gamma\frac{\rho^W-\gamma^W}{\rho-\gamma}\right)^2\\
  & \le \left(\left(\frac{2}{1-\gamma_0} + W\right)\gamma_0^W\right)^2
\end{split}
\end{equation}
Thus we have that:
\begin{align*}
   &\quad \sum_{t=1}^T \|N_{0|t}x_1 + \sum_{i=1}^{T-1}N_{i|t}B_dd_i|\|^2\\
    &\le c^2 \left[\left(\frac{2}{1-\gamma_0} + W\right)\gamma_0^W\right]^2\left(\|x\|^2 + \sum_{t=1}^T\|B_d d_t\|^2\right).
\end{align*}
We now bound
\begin{align*}
    &\quad \sum_{t=1}^{T-1}\|\sum_{j=1}^t\sum_{i=j}^{T-1} L_{(i,j)|t} B_de_{i|j}\|^2\\
    &\le  \sum_{t=1}^{T-1}\left(\sum_{j=1}^t\sum_{i=j}^{T-1} \|L_{(i,j)|t}\| \|B_de_{i|j}\|\right)^2\\
    &\le c^2 \sum_{t=1}^{T-1}\left(\sum_{j=1}^t\sum_{i=j}^{T-1} a_{i,j}^t \|B_de_{i|j}\|\right)^2
\end{align*}
where
\begin{equation*}
    \alpha_{i,j}^t = \left\{
    \begin{array}{ll}
        \rho^{t-2j+i-1}\gamma^W, & j\le t-1 \\
        \rho^{i-t}, & j=t\\
        0,&j\ge t+1
    \end{array}
    \right.
\end{equation*}
Applying Lemma \ref{lemma:auxillary}  we have that:
\begin{align*}
    &\quad\sum_{t=1}^{T-1}\left(\sum_{j=1}^t\sum_{i=j}^{T-1} a_{i,j}^t \|B_de_{i|j}\|\right)^2 \\&\le \sum_{j=1}^{T-1}\sum_{i=j}^{T-1}\left(\sum_{t=1}^{T-1}a_{i,j}^t\sum_{i^\prime \ge j^\prime}a_{i^\prime,j^\prime}^t\right)\|B_de_{i|j}\|^2
\end{align*}
For fixed $t$,
\begin{align*}
    \sum_{i^\prime \ge j^\prime}a_{i^\prime,j^\prime}^t &= \sum_{j^\prime=1}^{t-1} \sum_{i^\prime=j^\prime}^{T-1}\rho^{t-2j^\prime+i^\prime-1}\gamma^W + \sum_{i^\prime = t}^{T-1}\rho^{i^\prime - t}\\
    &\le \frac{1}{1-\rho} + \frac{\gamma^W}{1-\rho^2}
\end{align*}
For fixed $i,j$,
\begin{align*}
    \sum_{t=1}^{T-1}a_{i,j}^t &= \rho^{i-j} + \sum_{t=j+1}^{T-1}\rho^{t-2j+i-1}\gamma^W\\
    &=\rho^{i-j}\left(1+\frac{\gamma^W}{1-\rho}\right)
\end{align*}
Thus:
\begin{align*}
    &\quad \sum_{t=1}^{T-1}\|\sum_{j=1}^t\sum_{i=j}^{T-1} L_{(i,j)|t} B_de_{i|j}\|^2\\
    &\le  c^2\left(\frac{1}{1\!-\!\rho}\!+\!\frac{\gamma^W}{1\!-\!\rho^2}\right)\!\left(1\!+\!\frac{\gamma^W}{1\!-\!\rho}\right)\sum_{j=1}^{T-1}\sum_{i=j}^{T-1} \rho^{i\!-\!j}\|B_de_{i|j}\|^2
\end{align*}
Combining the above inequalities we have that:
\begin{small}
\begin{equation*}
\begin{split}
    &\qquad \textup{Regret}(\bm \pi)\\ &\le c^2\left[\frac{1}{1\!-\!\rho}(\gamma^W\!+\!\rho^W) \!+\! \gamma\frac{\rho^{\!W}\!-\!\gamma^{\!W}}{\rho\!-\!\gamma}\right]^2\left(\|x\|^2 + \sum_{t=1}^T\|B_d d_t\|^2\right)\\
   &+ c^2\left(\frac{1}{1\!-\!\rho}\!+\!\frac{\gamma^W}{1\!-\!\rho^2}\right)\left(1\!+\!\frac{\gamma^W}{1\!-\!\rho}\right)\sum_{j=1}^{T-1}\sum_{i=j}^{j+W} \rho^{i-j}\|B_de_{i|j}\|^2,
\end{split}
\end{equation*}
\end{small}
which completes the proof.
\end{proof}

\subsection{Proof of Lemma \ref{lemma:bound-K_t^i-difference}}\label{apdx:bound-K_t^i}
We define the following auxillary variables:
\begin{equation*}
   X_t^{d,s} = \Phi(s+1,t+1)^\top P_{s+1},~s\ge t
\end{equation*}
Similarly, we can define $\lbar{X}_{r|t}^{d,s}$.
\begin{equation*}
    \lbar{X}_{r|t}^{d,s} = \lbar{\Phi_t}(s+1,r+1)^\top \lbar{P}_{s+1|t}
\end{equation*}
It is not hard to verify that
\begin{align*}
    K_t^{d,s} &= (R_t + B_u^\top P_{t+1}B_u)^{-1}B_u^\top X_t^{d,s}\\
     \lbar{K}_i^{d,s} &= (R_t + B_u^\top \lbar{P}_{t+1|t}B_u)^{-1}B_u^\top \lbar{X}_{t|t}^{d,s}
\end{align*}
The following proposition relates the difference of $\lbar{X}_{t|t}^{d,s}-X_t^{d,s}$ to $\lbar{P}_{s|t} - P_{s}$.
\begin{prop}\label{prop:difference-X_t^i}
For $1 \le i \le W$, 
\begin{equation*}
\begin{split}
   &\lbar{X}_{t|t}^{d,s}-X_t^{d,s} =  \Phi(s+1,t+1)^\top (\lbar{P}_{s+1|t} - P_{s+1})\\
    &- \sum_{r=t+1}^{s} \Phi(r+1,t+1)^\top(\lbar{P}_{r+1|t} - P_{r+1})^\top B_u \\&\qquad\qquad\qquad\qquad~ (R_{r} + B_u^\top\lbar{P}_{r+1|t}B_u)^{-1}B_u^\top \bar{X}_{r|t}^{s}
\end{split}
\end{equation*}
\end{prop}
\begin{proof}
\begin{equation*}
\begin{split}
    \lbar{X}_{r|t}^{d,s} &= (A-B_u\lbar{K}_{r+1|t})^\top \lbar{X}_{r+1|t}^{d,s}\\
    X_{r}^{d,s} &= (A-B_uK_{r+1})^\top X_{r+1}^{d,s}, 
\end{split}
\end{equation*}
Thus,
\begin{equation*}
\begin{split}
    \lbar{X}_{t|t}^{d,s} \!- \!K_{t}^{d,s}\! &= ((A\!-\!B_u\lbar{K}_{t\!+\!1|t})^\top\! -\! (A\!-\!B_uK_{t\!+\!1})^\top)\lbar{X}_{t\!+\!1|t}^{d,s}\\
    &\qquad + (A-B_uK_{t+1})^\top (\lbar{X}_{t+1|t}^{d,s}-X_{t+1}^{d,s})\\
    &= -(\lbar{K}_{t+1|t}-K_{t+1})^\top B_u^\top \lbar{X}_{t+1|t}^{d,s}\\
    &\qquad+\Phi(t+2,t+1)^\top (\lbar{X}_{t+1|t}^{d,s}-X_{t+1}^{d,s})\\
    &=\cdots\\
    &=-\sum_{r=t+1}^s \Phi(r,t+1)(\lbar{K}_{r|t}-K_{r})^\top B_u^\top \lbar{X}_{r|t}^{d,s} \\
    &\qquad + \Phi(s+1,t+1)^\top (\lbar{X}_{s|t}^{d,s}-X_{s}^{d,s})
\end{split}
\end{equation*}
Since,
\begin{equation*}
\begin{split}
    &\lbar{K}_{r|t}-K_{r} = (R_{r}+B_u^\top \lbar{P}_{r+1|t}B_u)^{-1}B_u^\top\lbar{P}_{r+1|t}A \\&\qquad\qquad\qquad- (R_{r}+B_u^\top P_{r+1}B_u)^{-1}B_u^\top P_{r+1}A\\
    &= (R_{r}+B_u^\top \lbar{P}_{r+1|t}B_u)^{-1}B_u^\top(\lbar{P}_{r+1|t}-P_{r+1})A\\
    & \!+ \!\left((R_{r}\!+\!B_u^\top \lbar{P}_{r\!+\!1|t}B_u)^{\!-\!1}\!-\!(R_{r}\!+\!B_u^\top P_{r+1}B_u)^{\!-\!1}\!\right)B_u^{\!\top} P_{r\!+\!1}A\\
    &= (R_{r}+B_u^\top \lbar{P}_{r+1|t}B_u)^{-1}B_u^\top(\lbar{P}_{r+1|t}-P_{r+1})A\\
    &\quad -  (R_{r}+B_u^\top \lbar{P}_{r+1|t}B_u)^{-1}B_u^\top(\lbar{P}_{r+1|t}-P_{r+1})\\
    &\qquad\qquad\qquad\qquad\quad B_u(R_{r}+B_u^\top P_{r+1}B_u)^{-1}B_u^\top P_{r+1}A\\
    &= (R_{r}+B_u^\top \lbar{P}_{r\!+\!1|t}B_u)^{\!-\!1}B_u^\top(\lbar{P}_{r\!+\!1|t}-P_{r\!+\!1})(A-B_uK_{r})
\end{split}
\end{equation*}
Substitute the equation into the previous equation, we get
\begin{equation*}
\begin{split}
   &\lbar{X}_{t|t}^{d,s}-X_t^{d,s} =  \Phi(s+1,t+1)^\top (\lbar{P}_{s+1|t} - P_{s+1})\\
    &- \sum_{r=t+1}^{s} \Phi(r+1,t+1)^\top(\lbar{P}_{r+1|t} - P_{r+1})^\top B_u \\&\qquad\qquad\qquad\qquad~ (R_{r} + B_u^\top\lbar{P}_{r+1|t}B_u)^{-1}B_u^\top \bar{X}_{r|t}^{s}
\end{split}
\end{equation*}
which completes the proof.
\end{proof}

\begin{prop}\label{prop:difference-K_t^i}
\begin{align*}
    &\|\lbar{X}_{t|t}^{d,s}-X_t^{d,s}\| \le\\ &\frac{\lambda_{\max}(P_{\max})^4}{(1-\gamma)\lambda_{\min}(P_{\min})^2}\left(\|B_uR_{\min}^{-1}B_u^\top\| + 1\right) \rho^{s-t}\gamma^{W-s+t}
\end{align*}
\end{prop}
\begin{proof}
From Proposition \ref{prop:difference-X_t^i}, we have that:
\begin{align*}
    &\|\lbar{X}_{t|t}^{d,s}-X_t^{d,s}\| \le \|\Phi(s+1,t+1)^\top\| \|(\lbar{P}_{s+1|t} - P_{s+1})\|\\
    &+ \sum_{r=t+1}^{s} \|\Phi(r+1,t+1)^\top\|\|(\lbar{P}_{r+1|t} - P_{r+1})^\top\|\\&\qquad\qquad\qquad\quad  \|B_u (R_{r} + B_u^\top\lbar{P}_{r+1|t}B_u)^{-1}B_u^\top\| \|\bar{X}_{r|t}^{s}\|
\end{align*}
It is easy to show from definition that
$$\|\lbar{X}_{r|t}^{d,s}\| \le \lambda_{\max}(P_{\max})\cdot\tau\rho^{s-r}.$$
Thus, by applying Proposition \ref{prop:exponential-convergence-value-function}, we have that
\begin{align*}
    &\|\lbar{X}_{t|t}^{d,s}-X_t^{d,s}\| \le\tau\rho^{s-t}\gamma^{W-s+t}\frac{\lambda_{\max}(P_{\max})^2}{\lambda_{\min}(P_{\min})}
    \\
    &+\sum_{r=t+1}^{s}(\tau\rho^{r-t})\left(\gamma^{W-r+t}\frac{\lambda_{\max}(P_{\max})^2}{\lambda_{\min}(P_{\min})}\right)\\ &\qquad\qquad\qquad\|B_uR_{\min}^{-1}B_u^\top\|~\lambda_{\max}(P_{\max})\cdot\tau\rho^{s-r}\\
     &= \tau^2\frac{\lambda_{\max}(P_{\max})^3}{\lambda_{\min}(P_{\min})}~\|B_uR_{\min}^{-1}B_u^\top\|~\rho^{s-t}\sum_{r=t+1}^{s}\gamma^{W-r+t}\\
     &+ \tau\rho^{i-1}\gamma^{W+1-i}\frac{\lambda_{\max}(P_{\max})^2}{\lambda_{\min}(P_{\min})}\\
     &\le \frac{\tau^2}{1-\gamma}\frac{\lambda_{\max}(P_{\max})^3}{\lambda_{\min}(P_{\min})}\left(\|B_uR_{\min}^{-1}B_u^\top\|+1\right)\rho^{s-t}\gamma^{W-s+t}\\
     & = \frac{\lambda_{\max}(P_{\max})^4}{(1-\gamma)\lambda_{\min}(P_{\min})^2}\left(\|B_uR_{\min}^{-1}B_u^\top\|+1\right)\rho^{s-t}\gamma^{W-s+t},
\end{align*}
which proves the Proposition \ref{prop:difference-K_t^i}.
\end{proof}

We are now ready to prove Lemma \ref{lemma:bound-K_t^i-difference}
\begin{proof}{Lemma \ref{lemma:bound-K_t^i-difference}}
\begin{align*}
 &   K_t^{d,s} = (R_t + B_u^\top P_{t+1}B_u)^{-1}B_u^\top X_t^{d,s}\\
&     \lbar{K}_i^{d,s} = (R_t + B_u^\top \lbar{P}_{t+1|t}B_u)^{-1}B_u^\top \lbar{X}_{t|t}^{d,s}\\
&\Rightarrow~\lbar{K}_i^{d,s}-K_t^{d,s} \\
&=  (R_t + B_u^\top \lbar{P}_{t+1|t}B_u)^{-1}B_u^\top(\lbar{X}_{t|t}^{d,s} - X_t^{d,s})\\& + \left((R_t + B_u^\top \lbar{P}_{t+1|t}B_u)^{-1}-(R_t + B_u^\top P_{t+1}B_u)^{-1}\right)B_u^\top X_t^{d,s}\\
&=  (R_t + B_u^\top \lbar{P}_{t+1|t}B_u)^{-1}B_u^\top(\lbar{X}_{t|t}^{d,s} - X_t^{d,s})\\& - (R_t + B_u^\top \lbar{P}_{t+1|t}B_u)^{-1}B_u^\top(\lbar{P}_{t+1|t}-P_{t+1})\\&\qquad\qquad\qquad\qquad\qquad B_u(R_t + B_u^\top P_{t+1}B_u)^{-1}B_u^\top X_t^{d,s}\\
&\Rightarrow~\|\lbar{K}_i^{d,s}-K_t^{d,s}\| \\
&\le \frac{\|B_u\|}{\lambda_{\min}(R_{\min})}\|\lbar{X}_{t|t}^{d,s}-X_t^{d,s}\|\\& + \frac{\|B_u\|}{\lambda_{\min}(R_{\min})}\|\lbar{P}_{t+1|t}-P_{t+1}\|\|B_u^\top R_{\min}^{-1}B_u\|\|X_t^{d,s}\|\\
& \le \frac{2\|B_u\|\lambda_{\max}(P_{\max})^4\left(\|B_uR_{\min}^{-1}B_u^\top\|+1\right)}{(1-\gamma)\lambda_{\min}(R_{\min})\lambda_{\min}(P_{\min})^2}\rho^{s-t}\gamma^{W-s+t}\\
& = c_3\rho^{s-t}\gamma^{W-s+t}
\end{align*}
Additionally, it is not hard to verify that:
\begin{equation*}
    \lbar{K}_t - K_t = (\lbar{K}_t^{d,t}-K_t^{d,t})A.
\end{equation*}
Thus:
\begin{equation*}
    \|\lbar{K}_t - K_t\|\le c_3\|A\|\gamma^W = c_2 \gamma^W,
\end{equation*}
which completes the proof
\end{proof}

\subsection{Proof of Lemma \ref{lemma:bound-P_t-difference}}\label{apdx:bound-P_t}
The following proposition suggest that under Assumption \ref{assump:bounded-S-R}, the value function matrices $P_t$'s are bounded.
\begin{prop}\label{prop:bounded-value-function}
Suppose $\{Q_t, R_t\}_{t=1}^T$ satisfies Assumption \ref{assump:bounded-S-R}. Then its corresponding LQR value function matrices $P_t$ are bounded, i.e.
$$\exists P_{\max} \succ 0, \quad s.t. \quad Q_{\min} \preceq P_t \preceq P_{\max}$$
\end{prop}
\begin{proof}
It is obvious that $$P_t \succeq Q_{\min}, \quad, \forall t.$$

Furthermore,  by the definition of value function, we know that
\begin{align*}
    P_t(\{Q_i,R_i\}_{i=t}^T) \preceq P_t(\{Q_i^{\max},R_i^{\max}\}_{i=t}^T) \\
    \preceq P_1(\{Q_t^{\max}, R_t^{\max}\}_{t=1}^T).
\end{align*}
Thus let $P_{\max}:=P_1(\{Q_{\max}, R_{\max}\}_{t=1}^T)$
then $P_t \preceq P_{\max}$, which completes the proof.
\end{proof}
A key component in this subsection is the following invariant metric $\delta_{\infty}$ on positive definite matrices:
$$\delta_{\infty}(A,B):= \|\log(A^{-1/2}BA^{-1/2})\|.$$
Various properties of $\delta_{\infty}$ are given in Appendix \ref{apdx:Properties-of-Invariant-metric} of this paper and Appendix D in \cite{Krauth19}.

By Lemma \ref{lemma:invariant-metric-1} in Appendix \ref{apdx:Properties-of-Invariant-metric} and Proposition \ref{prop:bounded-value-function}, we can easily obtain the following corollary.

\begin{coro}
Given two sequences $\{Q_t, R_t\}_{t=1}^T, \{Q_t^\prime, R_t^\prime\}_{t=1}^T$, which both satisfy Assumption \ref{assump:bounded-S-R}. Then the distance between their corresponding LQR value function matrices $P_t, P_t^\prime$ are bounded, i.e.
$$ \delta_{\infty}(P_t^\prime, P_t),  \delta_{\infty}(P_t, P_t^\prime) \le \log\left(\frac{\lambda_{\max}(P_{\max})}{\lambda_{\min}(Q_{\min})}\right). $$
\end{coro}

Most importantly, by directly applying Lemma D.2 in \cite{Krauth19} and Lemma \ref{lemma:invariant-metric-2} in Appendix \ref{apdx:Properties-of-Invariant-metric}, we can immediately get the following proposition.

\begin{prop}\label{prop:exponential-convergence-value-function}
Given two sequences $\{Q_i, R_i\}_{i=t}^T, \{\bar{Q}_i, \bar{R}_i\}_{i=t}^T$ that satisfy Assumption \ref{assump:bounded-S-R}, and that:
\begin{equation*}
\begin{split}
  \bar{Q}_i& =   Q_i\\
  \bar{R}_i&= R_i, \quad i = t, t+1, \dots, t+W.
\end{split}
\end{equation*}
Then corresponding $P_i, \bar{P}_i$ satisfies:
\begin{equation*}
    \delta_{\infty}(\bar{P}_i, P_i) \le \gamma^{t+W+1-i} \log\left(\frac{\lambda_{\max}(P_{\max})}{\lambda_{\min}(Q_{\min})}\right),~ t\le i\le t+W,
\end{equation*}
where $\gamma = \frac{\lambda_{\max}(A^\top P_{\max} A)}{\lambda_{\min}(Q_{\min}) + \lambda_{\max}(A^\top P_{\max} A)}$.
Furthermore, by applying Lemma \ref{lemma:invariant-metric-2} in Appendix \ref{apdx:Properties-of-Invariant-metric}, we have:
$$\|\bar{P}_i - P_i\| \le \gamma^{t+W+1-i}~\frac{\lambda_{\max}(P_{\max})^2}{\lambda_{\min}(Q_{\min})}, \quad  t\le i\le t+W.$$
\end{prop}
\begin{proof}
Recall that
\begin{align*}
    P_{i} &= F_{Q_{i}, R_{i}}(P_{i+1}), \\
    \bar{P}_{i} &= F_{Q_{i},R_{i}}(\bar{P}_{i+1}),\quad t\le i\le t+W
\end{align*}
Define $\alpha := \max\{\lambda_{\max} (A^\top P_{i+1}A),\lambda_{\max} (A^\top \bar{P}_{i+1}A)\}$, Lemma D.2 in \cite{Krauth19} implies:
\begin{align*}
    \delta_{\infty}(\bar{P}_i,P_i) &\le \frac{\alpha}{\lambda_{\min} (Q_{\min}) + \alpha} \delta_{\infty}(\bar{P}_{i+1},P_{i+1})\\
    &\le \frac{\lambda_{\max}(A^\top P_{\max} A)}{\lambda_{\min}(Q_{\min}) + \lambda_{\max}(A^\top P_{\max} A)}\delta_{\infty}(\bar{P}_{i+1},P_{i+1})\\
    &= \gamma \delta_{\infty}(\bar{P}_{i+1},P_{i+1}), \quad t\le i\le t+W
\end{align*}
Applying this inequality recursively, we have:
\begin{align*}
    \delta_{\infty}(\bar{P}_i,P_i) 
    &\le \gamma \delta_{\infty}(\bar{P}_{i+1},P_{i+1})\\
    &\le \cdots\\
    &\le \gamma^{t+W+1-i} \delta_{\infty}(\bar{P}_{t+W+1}, P_{t+W+1})\\
    &\le \gamma^{t+W+1-i} \log\left(\frac{\lambda_{\max}(P_{\max})}{\lambda_{\min}(Q_{\min})}\right)
\end{align*}
Then by applying Lemma \ref{lemma:invariant-metric-2} in Appendix \ref{apdx:Properties-of-Invariant-metric}, we have:
$$\|\bar{P}_i - P_i\| \le \gamma^{t+W+1-i}~\frac{\lambda_{\max}(P_{\max})^2}{\lambda_{\min}(Q_{\min})}, \quad  t\le i\le t+W,$$
which completes the proof.
\end{proof}

\begin{proof}{(Lemma \ref{lemma:bound-P_t-difference})}

Lemma \ref{lemma:bound-P_t-difference} is simply a corollary of Proposition \ref{prop:exponential-convergence-value-function}
\end{proof}

\subsection{Exponential Stability}\label{apdx:exponential-stability}
In this section we will look into the exponential stability of both finite time LQR (Proposition \ref{prop:stability-LQR}) and MPC algorithm (Proposition \ref{prop:stability-MPC}).

\begin{proof}{(Proposition \ref{prop:stability-LQR})}

Let $V_t^{\textup{LQR}}(x):= x^\top P_t x$ be the value function for standard finite time horizon LQR problem. For arbitrary $t, i, x$, let
\begin{equation*}
\begin{split}
    x_t&:=x\\
    x_{t+j}&:=(A-B_uK_{t+j})\cdots(A-B_uK_{t+2})(A-B_uK_{t+1})x
\end{split}
\end{equation*}

According to Bellman optimality equation, we have that
\begin{align*}
    V_{t+1}^{\textup{LQR}}(x_{t+1}) &= V_t^{\textup{LQR}}(x_t) - x_t^\top Q_tx_t - x_t^\top K_t^\top R_tK_tx_t\\
    &\le V_t^{\textup{LQR}}(x_t) - x_t^\top Q_tx_t\\
    &=x_t^\top P_tx_t - x_t^\top Q_tx_t\\
    &\le \left(1-\frac{\lambda_{\min}(Q_t)}{\lambda_{\max}(P_t)}\right)x_t^\top P_tx_t\\
    &\le  \left(1-\frac{\lambda_{\min}(Q_{\min})}{\lambda_{\max}(P_{\max})}\right)V_t^{\textup{LQR}}(x_t)
\end{align*}
Similarly,
\begin{align*}
   V_{t\!+\!2}^{\textup{LQR}}(x_{t\!+\!2}) &\le \left(\!1-\frac{\lambda_{\min}(Q_{\min})}{\lambda_{\max}(P_{\max})}\!\right) V_{t\!+\!1}^{\textup{LQR}}(x_{t\!+\!1})\\
    &\cdots\\
    \Rightarrow \quad V_{t+i}^{\textup{LQR}}(x_{t+i}) &\le \left(1-\frac{\lambda_{\min}(Q_{\min})}{\lambda_{\max}(P_{\max})}\right)^i V_t^{\textup{LQR}}(x_{t})
\end{align*}
Since we have:
\begin{equation*}
\begin{split}
    &x_{t+i}^\top P_{\min}x_{t+i} \le V_{t+i}^{\textup{LQR}}(x_{t+i})\\ 
    &\le \left(1-\frac{\lambda_{\min}(Q_{\min})}{\lambda_{\max}(P_{\max})}\right)^i V_t^{\textup{LQR}}(x_{t})\\
    &\le \left(1-\frac{\lambda_{\min}(Q_{\min})}{\lambda_{\max}(P_{\max})}\right)^i x_t^\top P_{\max}x_t
\end{split}
\end{equation*}
$$\Rightarrow \|x_{t+i}\| \le \frac{\lambda_{\max}(P_{\max})}{\lambda_{\min}(P_{\min})} \left(1-\frac{\lambda_{\min}(Q_{\min})}{\lambda_{\max}(P_{\max})}\right)^i\|x_t\|$$
where $x_t = x,~x_{t+i} = (A-B_uK_{t+i})\cdots(A-B_uK_{t+2})(A-B_uK_{t+1})$. The inequality holds for arbitrary $x$, thus we have:
\begin{align*}
    \|(A-B_uK_{t+i})\cdots(A-B_uK_{t+2})(A-B_uK_{t+1})\| \\
    \le \sqrt{\frac{\lambda_{\max}(P_{\max})}{\lambda_{\min}(P_{\min})}} \left(\sqrt{
1-\frac{\lambda_{\min}(Q_{\min})}{\lambda_{\max}(P_{\max})}}\right)^i,
\end{align*}
which proves the proposition.
\end{proof}
The proof for exponential stability of MPC algorithm is similar to the proof provided above.
\begin{proof}{(Proposition \ref{prop:stability-MPC})}

First we define a time-varying Lyapunov function:
\begin{align*}
    &V_t^{\textup{MPC}}(x):= x^\top \lbar{P}_{t|t} x\\
    &= \min \sum_{s=t}^{t+W}(x_s^\top Q_s x_s + u_s^\top R_s u_s) + x_{t+W+1}^\top P_{\max} x_{t+W+1}\\
    &\quad s.t.\quad x_{s+1} = Ax_s + B_u u_s\\
    &\qquad\qquad~ x_t = x
\end{align*}
Applying Proposition \ref{prop:bounded-value-function} we can show that:
\begin{equation*}
    Q_{\min}\preceq \lbar{P}_{t|t} \preceq P_{\max}.
\end{equation*}
For arbitrary $t,j,x$, let:
\begin{align*}
    x_t&:= x\\
    x_{t+j}&:= (A-B_u\lbar{K}_{t+j})\cdots(A-B_u\lbar{K}_{t+1})x\\
    &= \lbar{\Phi}^{\textup{MPC}}(t+j+1, t+1)x
\end{align*}
By the definition of $\lbar{K_{t}}, P_{\max}$, we have that:
\begin{align*}
    V_t^{\textup{MPC}}(x_t) =\! \min_{x_{s\!+\!1}\!=\!Ax_s\!+\!B_uu_s}\!\left(\sum_{s\!=\!t}^{t\!+\!W}(x_s^\top Q_s x_s \!+\! u_s^\top R_s u_s)\right. \\\left.+ x_{t+W+1}^\top P_{\max} x_{t+W+1}\right)\\
    = \min_{x_{s+1}=Ax_s+B_uu_s}\left(\sum_{s=t}^{t+W}(x_s^\top Q_s x_s + u_s^\top R_s u_s)\right. \\\left.+ \sum_{s=t+W+1}^{+\infty}(x_s^\top Q_{\max} x_s + u_s^\top R_{\max} u_s) \right)\\
    = x_t^\top Q_tx_t + x_t^\top\lbar{K}_t^\top R_t \lbar{K}_tx_t\qquad\qquad\qquad\qquad\\
    +\min_{x_{s+1}=Ax_s+B_uu_s}\left(\sum_{s=t+1}^{t+W}(x_s^\top Q_s x_s + u_s^\top R_s u_s)\right. \\\left.+ \sum_{s=t+W+1}^{+\infty}(x_s^\top Q_{\max} x_s + u_s^\top R_{\max} u_s) \right)\\
    \ge x_t^\top Q_tx_t + x_t^\top\lbar{K}_t^\top R_t \lbar{K}_tx_t\qquad\qquad\qquad\qquad\\
    +\min_{x_{s+1}=Ax_s+B_uu_s}\left(\sum_{s=t+1}^{t+W+1}(x_s^\top Q_s x_s + u_s^\top R_s u_s)\right. \\\left.+ \sum_{s=t+W+2}^{+\infty}(x_s^\top Q_{\max} x_s + u_s^\top R_{\max} u_s) \right)\\
    = x_t^\top Q_tx_t + x_t^\top\lbar{K}_t^\top R_t \lbar{K}_tx_t + V_{t+1}^{\textup{MPC}}(x_{t+1})\quad
\end{align*}
Thus,
\begin{align*}
    V_{t+1}^{\textup{MPC}}(x_{t+1}) \le V_t^{\textup{MPC}}(x_t) - x_t^\top Q_tx_t\\
    \le \left(1-\frac{\lambda_{\min}(Q_{\min})}{\lambda_{\max}(P_{\max})}\right)V_t^{\textup{MPC}}(x_t)
\end{align*}
Similarly,
\begin{align*}
    V_{t+j}^{\textup{MPC}}(x_{t+j}) 
    \le \left(1-\frac{\lambda_{\min}(Q_{\min})}{\lambda_{\max}(P_{\max})}\right)^j V_t^{\textup{MPC}}(x_t)\\
    \Rightarrow ~ \|x_{t+j}\|^2 \le \frac{\lambda_{\max}(P_{\max})}{\lambda_{\min}(Q_{\min})}\left(1-\frac{\lambda_{\min}(Q_{\min})}{\lambda_{\max}(P_{\max})}\right)^j\|x_t\|^2,
\end{align*}
which completes the proof.
\end{proof}
\subsection{Properties of Invariant Metric}\label{apdx:Properties-of-Invariant-metric}
\begin{lemma}\label{lemma:invariant-metric-1}
Suppose $A,B$ are positive definite, and
$$L \preceq A,B\preceq U,$$
where $L,U$ are both positive definite matrices. Then:
\begin{equation*}
    \delta_{\infty}(A,B) \le \log\left(\frac{\lambda_{\max}(U)}{\lambda_{\min}(L)}\right)
\end{equation*}
\end{lemma}
\begin{proof}
\begin{equation*}
\begin{split}
    &\delta_{\infty}(A,B) = \|\log(A^{-1/2}BA^{-1/2})\|\\
    &= \max\left\{\left|\log(\lambda_{\max}\left(A^{-1/2}BA^{-1/2}\right))\right|,\right.\\&~\qquad\qquad\left.\left|\log(\lambda_{\min}\left(A^{-1/2}BA^{-1/2}\right))\right|\right\}
\end{split}
\end{equation*}
Since
\begin{align*}
    &A^{-1/2}BA^{-1/2} \preceq A^{-1/2}\left(\lambda_{\max}(U)I\right)A^{-1/2}\\& = \lambda_{\max}(U)A^{-1}
    \preceq \lambda_{\max}(U)L^{-1}\preceq\frac{\lambda_{\max}(U)}{\lambda_{\min}(L)} I
\end{align*}
\begin{align*}
    &A^{-1/2}BA^{-1/2} \succeq A^{-1/2}\left(\lambda_{\min}(L)I\right)A^{-1/2} \\&= \lambda_{\max}(L)A^{-1}
    \succeq \lambda_{\min}(L)U^{-1}\preceq\frac{\lambda_{\min}(L)}{\lambda_{\max}(U)} I
\end{align*}
Then we have that,
\begin{equation*}
\begin{split}
     -\log\left(\frac{\lambda_{\max}(U)}{\lambda_{\min}(L)}\right)\le \log\left(\lambda_{\min}\left(A^{-1/2}BA^{-1/2}\right)\right)\\\le\log\left(\lambda_{\min}\left(A^{-1/2}BA^{-1/2}\right)\right)\le \log\left(\frac{\lambda_{\max}(U)}{\lambda_{\min}(L)}\right)
\end{split}
\end{equation*}
which completes the proof.
\end{proof}

\begin{lemma}\label{lemma:invariant-metric-2}
Suppose $A,B$ are positive definite, and
$$A \preceq U, \quad \delta_{\infty}(A,B) \le c,$$
where $U$ is a positive definite matrix. Then,
$$\|A - B\| \le \lambda_{\max}(U) \frac{e^c-1}{c}\delta_{\infty}(A,B).$$
\end{lemma}
\begin{proof}
\begin{equation*}
\begin{split}
    \|A-B\| &\le \|A^{1/2}(I-A^{-1/2}BA^{-1/2})A^{1/2}\|\\
    &\le \|A^{1/2}\|^2 \|I-A^{-1/2}BA^{1/2}\|\\
    &\le\lambda_{\max}(U) \|I-A^{-1/2}BA^{1/2}\|
\end{split}
\end{equation*}
Since
\begin{equation*}
\begin{split}
    e^{\!-\!\delta_{\infty}(A,B)}I \!-\! I\!\preceq\!\lambda_{\min}(A^{\!-\!1/2}BA^{1/2})I\!-\!I\!\preceq \!A^{\!-\!1/2}BA^{1/2}\!-\!I\\\preceq\lambda_{\max}(A^{-1/2}BA^{1/2})I-I\preceq e^{\delta_{\infty}(A,B)}I - I
\end{split}
\end{equation*}
Thus
\begin{align*}
    \|I-A^{-1/2}BA^{1/2}\| &\le \max\left\{e^{\delta_{\infty}(A,B)}\!-\!1, 1\!-\! e^{-\delta_{\infty}(A,B)} \right\}\\
    &\le \max\left\{e^{\delta_{\infty}(A,B)}-1, \delta_{\infty}(A,B)\right\}.
\end{align*}
It is easy to verify that $e^x - 1 \le \frac{e^c-1}{c}x$ for $0\le x\le c$.
Thus, we have
$$e^{\delta_{\infty}(A,B)}-1 \le \frac{e^c-1}{c}\delta_{\infty}(A,B),$$
which completes the proof.
\end{proof}
\subsection{Others}\label{apdx:others}
\begin{lemma}\label{lemma:auxillary}
$y_1, \dots, y_n \in \mathbb{R}$, and
$$\alpha_t = a_1^t y_1 + \dots + a_n^ty_n, ~t = 1,2,\dots,T,$$
where $a_i^t \ge 0, i = 1,2,\dots, n$. Then:
\begin{align*}
    \sum_{t=1}^T\alpha_t^2 &\le\sum_{i}\left(\left\{\sum_{t=1}^T a_i^t\left(\sum_{j=1}^n a_j^t\right)\right\} y_i^2\right)\\
    &\le \max_{i}\left\{\sum_{t=1}^T a_i^t\left(\sum_{j=1}^n a_j^t\right)\right\} \left(\sum_{i=1}^n y_i^2\right)
\end{align*}
\end{lemma}
\begin{proof}
Define:
$$a_t:= [a_1^t, \dots, a_n^t]^\top, \quad y:= [y_1, \dots, y_n]^\top,$$ then
$$\sum_{t=1}^T\alpha_t^2 = y^\top\left(\sum_{t=1}^Ta_t^\top a_t\right)y.$$
Let
$$A:= \sum_{t=1}^Ta_t^\top a_t.$$
We have: $A_{ij} =\sum_{t=1}^T a_i^ta_j^t \ge 0$.
Thus $$diag\left\{\left\{\sum_j A_{ij}\right\}_i\right\} - A$$ is diagonally dominant, and thus,
\begin{align*}
    &diag\left\{\left\{\sum_j A_{ij}\right\}_i\right\} - A\succeq 0\\
    \Rightarrow &diag\left\{\left\{\sum_j A_{ij}\right\}_i\right\}\succeq A.
\end{align*}
Thus,
\begin{align*}
   \sum_{t=1}^T\alpha_t^2 &= y^\top Ay \\
   &\le  y^\top diag\left\{\left\{\sum_j A_{ij}\right\}_i\right\} y\\
   &= \sum_{i}\left(\left\{\sum_{t=1}^T a_i^t\left(\sum_{j=1}^n a_j^t\right)\right\} y_i^2\right)\\
    &\le \max_{i}\left\{\sum_{t=1}^T a_i^t\left(\sum_{j=1}^n a_j^t\right)\right\} \left(\sum_{i=1}^n y_i^2\right)   
\end{align*}
\end{proof}

\end{document}